\documentclass[10pt]{article}

\usepackage{url}
\usepackage{mathtools}
\usepackage{amssymb}
\usepackage{amsthm}
\usepackage{empheq}
\usepackage{latexsym}
\usepackage{enumitem}
\usepackage{eurosym}
\usepackage{dsfont}
\usepackage{appendix}
\usepackage{color} 
\usepackage[unicode]{hyperref}
\usepackage{frcursive}
\usepackage[utf8]{inputenc}
\usepackage[T1]{fontenc}
\usepackage{geometry}
\usepackage{multirow}
\usepackage{todonotes}
\usepackage{lmodern}
\usepackage{anyfontsize}
\usepackage{stmaryrd}
\usepackage{natbib}
\usepackage{cleveref}

\setcitestyle{numbers,open={[},close={]}}

\definecolor{red}{rgb}{0.7,0.15,0.15}
\definecolor{green}{rgb}{0,0.5,0}
\definecolor{blue}{rgb}{0,0,0.7}
\hypersetup{colorlinks, linkcolor={red},citecolor={green}, urlcolor={blue}}
			
\makeatletter \@addtoreset{equation}{section}

\newtheorem{theorem}{Theorem}[section]

\newtheorem{lemma}[theorem]{Lemma}
\newtheorem{proposition}[theorem]{Proposition}

\newtheorem{definition}[theorem]{Definition}

\DeclareUnicodeCharacter{014D}{\=o}
\def \E{\mathbb{E}}

\def \P{\mathbb{P}}

\def \R{\mathbb{R}}

\def\Ac{{\cal A}}

\def\Fc{{\cal F}}

\def\1{\mathbf{1}}
\def\dd{\mathrm{d}}

\def \proof{{\noindent \bf Proof.\quad}}
\def \ep{\hbox{ }\hfill{ ${\cal t}$~\hspace{-5.1mm}~${\cal u}$}}

\def\b*{\begin{eqnarray*}}
\def\e*{\end{eqnarray*}}

\def\be{\begin{eqnarray}}
\def\ee{\end{eqnarray}}

\begin{document}


\title{\bf Optimal resource allocation for maintaining system solvency}
\author{Gaoyue Guo\thanks{Université Paris-Saclay CentraleSupélec, Laboratoire MICS and CNRS FR-3487, gaoyue.guo@centralesupelec.fr. This work was supported by the grant ANR-25-CE40-0714 (MATH-SPA).}
            \and 
            Nizar Touzi\thanks{New York University, 
            Tandon School of Engineering, nizar.touzi@nyu.edu. This author is partially supported by NSF grant No. DMS-2508581 and the Chair Finance and Sustainable Development.}
            \and 
            Wenpin Tang\thanks{Columbia University, IEOR, wt2319@columbia.edu. This author is supported by NSF CAREER Award DMS-2538791, the Tang Family Assistant
Professorship and a Columbia-CityU/HK collaborative project that is supported by InnoHK
Initiative, The Government of the HKSAR and the AIFT Lab.}
            }

\date{\today}

\maketitle

\begin{abstract}
We study two optimal allocation problems for a system of independent Brownian agents whose states evolve under a limited shared control. At each time, a unit of resource can be divided and allocated across components to increase their drifts, with the objective of maximizing either (i) the probability that all components avoid ruin, or (ii) the expected number of components that avoid ruin. 
We identify drift thresholds separating trivial and nontrivial regimes,
and derive the associated Hamilton--Jacobi--Bellman equations on the positive orthant with mixed boundary conditions at the absorbing boundary and at infinity.
We also establish the existence, uniqueness, and regularity of a bounded classical solution and a verification theorem linking the PDE to the stochastic control value function. 
Finally, we prove a conjecture on the optimality of a socialistic allocation rule -- the push-the-laggard strategy.
It is optimal for the all-survive value function, while it is suboptimal for the count-survivors criterion.
\end{abstract}


\section{Introduction}\label{sec:introduction}

We study a stochastic allocation problem for a system of independent Brownian agents whose states evolve under a shared control budget.
An agent is annihilated when her state level hits zero, 
while the planner can distribute one unit of resource per unit time across the agents to increase their state levels.
The objective is to allocate the resource dynamically so as to maximize long-run survival. 
Two natural criteria lead to qualitatively different control problems:
\begin{enumerate}
\item[(1)]
maximize the probability that all agents survive forever;
\item[(2)]
maximize the expected number of agents that survive forever. 
\end{enumerate}
These two objectives capture, respectively, a system-wide robustness criterion and an aggregate survival criterion.

Many real-world allocation problems share such a common tension: 
a planner can distribute a limited resource among several units whose situations evolve randomly, and the planner’s objective may reflect either efficiency, meaning maximizing the number of units that do well, or robustness, meaning avoiding catastrophic failure of the system as a whole.
A vivid example comes from emergency departments: 
hospitals allocate scarce resources such as physician time, ICU beds, ventilators, and diagnostic capacity to a finite number of patients, each with uncertain health trajectory.
Operations research models \cite{Ahsan2019EDreview, Duma2023ED} emphasize that decisions must be made in real time under uncertainty, and that different prioritization rules can strongly affect system performance and patient outcomes.
From a normative viewpoint, triage guidelines balance several ethical principles \cite{Zeneli2021TriageValues}, such as ``save the most lives'', ``protect the worst-off'', and ``treat people equally'', and these principles may conflict in scarce-resource settings \cite{Persad2009Scarce}.

A second viewpoint comes from social welfare design: should a planner concentrate effort on those closest to failure, following a maximin or Rawlsian intuition \cite{PatraoNeves2022EthicalCriteria}, or instead allocate effort where it yields the greatest aggregate benefit, following a more utilitarian logic?
This question is closely related to an interpretation proposed by McKean and Shepp \cite{MS2006},
who analyzed competing allocation principles within a stochastic framework for distributing resources across (two) companies.
They compared two distinct government tax policies for corporations: 
the republican policy gives tax breaks to the richer companies,
while the democratic policy gives tax breaks to the weaker companies in the hope of keeping them alive and thereby reducing unemployment.

A central (and natural) question is:
\begin{center}
{\it Does allocating resources to the weakest agent/entity yield the highest marginal benefit?}
\end{center}
Such an allocation rule is referred to as the {\em push-the-laggard strategy} \cite{TT2018}. 
In the emergency department analogy, this resembles a sickest-first rule; in a social welfare analogy, it resembles a prioritize-the-least-advantaged rule.
Our main goal is to show that the answer to the above question depends sharply on the objective: 
prioritizing the weakest is optimal for a stringent system-survival criterion, but may fail for a maximize-the-number-of-survivors criterion,
extending \cite{MS2006} for any finite number of agents/entities.

Technically, we analyze this problem through its Hamilton--Jacobi--Bellman (HJB) equation on the positive orthant, supplemented with absorbing boundary conditions at zero and recursive asymptotic conditions at infinity. 
Our first result identifies drift thresholds distinguishing between trivial and nontrivial regimes. 
In the nontrivial regime, we solve the delicate question of identifying the boundary condition at infinity. 
We prove the existence, uniqueness, and regularity of bounded classical solutions, and establish a verification theorem linking the value function and the HJB equation.
Our second result reveals a sharp dichotomy in the structure of optimal controls: the push-the-laggard strategy is optimal for the all-survive criterion, but fails in general for the survivor-count criterion.  
We prove these results by fine analysis of the HJB equation near the origin, which is the corner of the positive orthant.


\medskip
\noindent
{\bf Related literature}: 
The push-the-laggard strategy is an instance of the rank-dependent diffusions,
which have been extensively studied in the literature, see e.g., \cite{BB22, BB24, BFK05, CDSS19, IKS13, JR13, PP08, ST17}.
It was proved to be the optimal drift control for the all-survive criterion \cite{MS2006},
and be suboptimal for the survivor-count criterion \cite{Grandits2019RuinBM} when $n = 2$.
When $n \longrightarrow \infty$, \cite{Aldous, TT2018} showed that the push-the-laggard strategy is asymptotically optimal for the survivor-count criterion.
See also \cite{Harrison2025DriftControlRBM} for recent work on drift control in high dimensions (i.e., $n$ is large) for operations management problems.
Besides, rank-dependent diffusions appeared in the study of stochastic games \cite{GTX22}
and control of McKean-Vlasov dynamics \cite{BHJ25, BGTZ} in the context of financial systemic risk.
Recently, \cite{AT24, BBE23} studied the connection between rank-dependent diffusions and load balancing in queueing theory.

\medskip
\noindent
{\bf Organization of the paper}: 
Section \ref{sec:pb} introduces the stochastic control problems, the push-the-laggard strategy, and the associated HJB equation, and Section \ref{sec:main} states the main results.
Section \ref{sec:largeparticles} proves the recursive boundary conditions at infinity, including the capped survivor-count recursion needed for $U^n$ when $b>-1$.
Section \ref{sec:proof-wellposed} establishes the HJB characterization, regularity, uniqueness, and verification result for $V^n$ and for the capped values $U^{n,r}$.
We then prove that the push-the-laggard strategy fails for the survivor-count criterion in Section \ref{sec:conjectureU}, and that it is optimal for the all-survive criterion in Section \ref{sec:conjectureV}.
The appendix contains the auxiliary one-dimensional Brownian estimates used in the large-particle and verification arguments.

\section{Problem formulation and properties}\label{sec:pb}

\subsection{The save-the-most and save-all resource allocation problems}\label{ssec:pb}

We consider two allocation problems formulated on some filtered probability space $(\Omega, \mathcal F, \mathbb F=(\mathcal F_t)_{t\ge 0}, \mathbb P)$ satisfying the usual conditions and supporting a standard $\mathbb F$-Brownian motion $W=(W^1,\ldots,W^n)$ on $\R^n$, for some $n\ge 1$.  Denote by $\mathcal A_n$ the collection of $\mathbb F$-progressively measurable processes $\Phi=\big(\phi_t=(\phi^1_t,\ldots, \phi^n_t)\big)_{t\ge 0}$ taking values in the simplex $A_n$ of $\R^n$:
$$
\phi_t\in A_n:=\{a\in\R_+^n: \mathbf{1}^{\!n}\cdot a \le 1\},
~~\mbox{for all}~~t\ge 0,
$$
where $\mathbf{1}^{\!n}\in \mathbb R^n$ stands for the unit vector with components $\mathbf{1}^n_i=1$, $i=1,\ldots,n$. 

Given some fixed parameter $b\in\R$, we define for all $\Phi\in\mathcal A_n$ the controlled process $X^{\Phi}=\big(X^\Phi_t=(X^{\Phi,1}_t,\ldots, X^{\Phi,n}_t)\big)_{t\ge 0}$ by
$$
\dd X^{\Phi}_t:=(b\mathbf{1}^{\!n}+\phi_t)\,\dd t+ \dd W_t,~~ t\ge 0,
$$
and we introduce the hitting time
$$
\tau^\Phi_i:=\inf\big\{t\ge 0:\, X^{\Phi,i}_t\le 0\big\},\qquad i=1,\ldots, n.
$$
We interpret $X^{\Phi}_t$ as the vector process of reserve level of the system units taking values in the positive orthant $\R_+^n$. The interior domain 
$$E_n:=(0,\infty)^n
$$ 
represents the solvency region of the system, while  absorption at $0$ of any entry of $X$ represents failure of the corresponding particle, such as death, default, or bankruptcy. The control $\phi^i_t$ interprets as the instantaneous assistance allocated to unit $i$ by the regulator, subject to a unit total budget\footnote{All the results presented below remain valid if the unit budget is replaced by any positive budget.}.  

Throughout this paper, we consider the following stochastic control problems: 
\begin{eqnarray}\label{def:VU}
V^n(x)
:=
\sup_{\Phi\in \mathcal A_n}\, \mathbb E_x\Big[\prod_{i=1}^n {\mathds 1}_{\{\tau_i^\Phi=\infty\}} \Big], 
~~\mbox{and}~~
U^n(x)
:=
\sup_{\Phi\in \mathcal A_n}\, \mathbb E_x\Big[\sum_{i=1}^n {\mathds 1}_{\{\tau_i^\Phi=\infty\}}\Big]. 
\end{eqnarray}
with $\mathbb E_x:=\mathbb E[\cdot|X^\Phi_0=x]$ denoting the expectation operator conditional on the state process starting at the position $x\in \mathbb R^n_+$.
The functional $V^n$ corresponds to a system-reliability objective: maximize the probability that every unit survives forever. The functional $U^n$ corresponds to a survivor-count objective: maximize the expected number of units that survive forever. These two objectives mirror, respectively, a no-one-left-behind policy and a save-as-many-as-possible policy in triage ethics \cite{Persad2009Scarce,Zeneli2021TriageValues}.

We observe that the problem becomes trivial for $n=1$, as the unit process $\phi^{1}\equiv 1\in \Ac_1$ is an optimal strategy.  Hence $V^1(z)=U^1(z)=\mathbb P[z+ (1+b)t + W^1_t >0,\, \forall t\ge 0]$ can be explicitly expressed as
\begin{equation}\label{def:n=1}
V^1(z)=U^1(z)=H_{b+1}(z),\qquad \mbox{where } H_{a}(z):=1-\mathrm{e}^{-2a^+z}
~~\mbox{for all}~z\ge 0,
\end{equation}
 with $a^+: = \max (a,0)$.

Our main objective in the paper is to provide a complete characterization of optimal allocation strategies in the non-trivial case $n\ge 2$. We shall adopt the notations
$$ 
\max(x):=\max\{x_1,\ldots, x_n\}
~~\mbox{and}~~
\min(x):=\min\{x_1,\ldots, x_n\},
~~\mbox{for all}~~
x=(x_1,\ldots, x_n)\in\mathbb R^n.
$$
Moreover, for any $I\subset\{1,\ldots,n\}$ with cardinality $|I|$ and with complement set $I^c:=\{1,\ldots,n\}\setminus I$, we shall denote throughout this paper $(x^{I}, x^{-I}):=x$,  where
\begin{equation}\label{notations}
x^{I}:=(x_j)_{j\in I}\in \mathbb R_+^{|I|},
~~, x^{-I}:=x^{I^c}\in \mathbb R_+^{|I^c|},
~~ \mathbf{0}(x):=\{i:x_i=0\},
~~x^{-0}:=x^{-\mathbf{0}(x)}\in E_{|\mathbf{0}(x)^c|}. 
\end{equation}
A natural conjecture is that, whenever one coordinate is strictly smaller than another, the value function is steeper in the corresponding direction, meaning that the marginal value of assistance is larger for the weakest unit. Equivalently, an optimal feedback would allocate effort to the smallest coordinate. This conjecture formalizes the question whether a prioritize-the-worst-off policy is optimal.

\begin{definition}\label{def:push}
The push-the-laggard strategy $\overline{\Phi}$ is defined by the tie-breaking rule:  
\[
\overline\phi^i_t:={\mathds 1}_{\{ i=I(X_t) \}},~t\ge 0,\qquad
~\text{with} ~I(x):=\min\{1\le j \le n: x_j=\min(x^{-0}) \}
~\mbox{for all}~x\in\R_+^n.
\]
\end{definition}
The induced system of SDEs $X^{\bar\Phi}$ has rank-dependent drifts and corresponds to an Atlas-type model whose existence and uniqueness are guaranteed by standard results in rank-based diffusions, see e.g.   \cite{IKS13, PP08, zvonkin1974transformation}. 
Our objective in this paper is to examine whether the following statement holds for either problems $U^n$ and $V^n$. 

\paragraph*{\rm\bf Conjecture 1:} {\it The push-the-laggard strategy $\overline{\Phi}$ is optimal.}

\vspace{5mm}
The above Conjecture 1 has been studied in the previous literature in the setting $b=0$:
\begin{itemize}
\item For $n=2$, McKean and Shepp \cite{MS2006} derived the explicit expression
$$
V^2(x)=1-\mathrm{e}^{-2\min(x)}-2\min(x)\mathrm{e}^{-\mathbf{1}^{\!2}\cdot x},
$$
which proves that Conjecture 1 holds for $V^2$ by direct verification. Moreover, they presented numerical evidence that the conjecture fails for $U^2$, which was  shown later by Grandits \cite{Grandits2019RuinBM}.  See also Grandits \cite{Grandits2025Ruin2D} for the heterogeneous case with $n=2$.

\item When $n\longrightarrow\infty$, Tang and Tsai \cite{TT2018} proved that Conjecture 1 is ``asymptotically true'' for $U^n$.
\end{itemize}
The parameter $b$, which was not addressed in the previous literature, will be shown to play a substantial role in the analysis of the corresponding HJB equation and in the study of the conjecture.

\subsection{The HJB equation}
\label{sec:HJBverif}

Our analysis builds on the PDE characterization of the control problems \eqref{def:VU}, which involves the Hamiltonian
\begin{equation}\label{def:H}
H(z):=\sup_{a\in A_n}a\cdot z=\max(z),
~\mbox{with maximizer}~
\hat a_i(z):=\mathds 1_{\{\max(z)=z_i>z_j \mbox{\tiny \color{blue} for all } j<i\}},
~~i=1,\ldots,n,
\end{equation}
where we set by convention $z_0=0$. Then, by standard control theory,  the HJB equation corresponding to both problems $V^n$ and $U^n$ can be written in terms of the Hamiltonian  as:
\begin{equation}\label{eq:hjb}
\frac{1}{2}\Delta w + b\, \mathbf{1}^{\!n}\!\cdot\! \nabla w + \max(\nabla w) = 0
~~\text{on}~~E_n.
\end{equation}
Of course, this equation needs to be complemented with appropriate boundary and growth conditions in order to provide a complete characterization of the value functions $V^n$ and $U^n$. This crucial aspect will be studied in the next subsection. Modulo this characterization, it is well-known from standard optimal control theory that the optimal strategies for both problems $V^n$ and $U^n$ are defined as the maximizers of the Hamiltonian $H$ introduced in \eqref{def:H}. Consequently, using the notations of Definition \ref{def:push}, the optimality of the push-the-laggard strategy of Conjecture 1 rewrites as follows.

\paragraph*{\rm\bf Conjecture 2:} {\it $w\in C^1(E_n)$ and $\max\big(\nabla w(x)\big)= \partial_{x_{_{{\!I\!(\!x\!)}}}}w(x)$, for all $x\in E_n$.}

\subsection{First  properties}

In this paragraph, we {\color{blue} provide} two simple results which will be used throughout this paper.

\begin{proposition}\label{prop:firststeps}
The maps $U^n,V^n$ are symmetric, componentwise nondecreasing on $\mathbb R^n_+$, and:

\begin{itemize}
\item [{\rm (i)}] $\prod_{i=1}^n H_{b+\frac1n}(x_i)\le V^n(x)\le U^n(x) \le \sum_{i=1}^n H_{b+1}(x_i)$, for all $x\in\R_+^n$;

\vspace{2mm}
\item [{\rm (ii)}] $\big|V^n(x)-V^n(y)\big|\vee\big|U^n(x)-U^n(y)\big|
\le n H_{b+1}(|x-y|_\infty)\le 2n (b+1)^+|x-y|_\infty$,  for all $x,y\in\R_+^n$. In particular $V^n$ and $U^n$ are Lipschitz on $\R_+^n$.
\end{itemize}
\end{proposition}

\begin{proof}
(i) The symmetry and the component-wise monotonicity, together with the inequalities $0\le V^n\le U^n$, follow from the definition of  $V^n$ and $U^n$. To complete the proof of (i), we observe that $X_t^{\Phi,i}=x_i+\int_0^t (b+\phi_s^i)\,\dd s+W_t^i
\le x_i+(b+1)t+W_t^i,$ for all $t\ge 0$, $i=1,\ldots,n$, and all admissible control $\Phi\in\mathcal A_n$. Then $\{\tau_i^\Phi=\infty\}
\subset
Q_i:=\big\{x_i+(b+1)t+W_t^i>0,\ \forall t\ge 0\big\}$, and therefore 
\b*
\mathbb E_x\Big[\sum_{i=1}^n {\mathds 1}_{\{\tau_i^\Phi=\infty\}}\Big]
&\le&
\sum_{i=1}^n\mathbb P_x[Q_i]
=
\sum_{i=1}^n H_{b+1}(x_i), 
\e* 
by the independence of the events $Q_i$ inherited from that of the Brownian motions $W^i$. 
Taking the supremum over $\Phi\in\mathcal A_n$ yields the required upper bounds for  $U^n$. On the hand, using the constant control $\Phi^0=(\frac1n,\ldots,\frac1n)\in\mathcal A_n$, it follows from the independence of the $W^i$'s that $$
V^n(x)\ge
\prod_{i=1}^n
\mathbb P_x\Big[x_i+(b+\frac1n)t+W_t^i>0,\ \forall t\ge 0\Big]
=
\prod_{i=1}^n H_{b+\frac1n}(x_i).
$$
(ii) For $w$ denoting either \(V^n\) or \(U^n\), the required estimate is implied by the following claim:
\begin{equation}\label{eq:uniform-continuity-value}
0\le w(x+\delta\mathbf 1^{\!n})-w(x)
\le n H_{b+1}(\delta),
~\text{for all}~x\in\mathbb R_+^n~\text{and}~\delta>0,
\end{equation}
that we now prove. For any admissible control \(\Phi\in\mathcal A_n\),  consider two
processes driven by the same Brownian motion and the same control, one starting
from \(x\), the other from \(x+\delta\mathbf 1^{\!n}\), so that \(X_t^{\delta,i}=X_t^i+\delta
\) for $i=1,\ldots,n$.  Denote the corresponding
hitting times by \(\tau_i\) and \(\tau_i^\delta\). 
A straightforward verification yields
\begin{equation}\label{ineq:cont}
\max\Big(\prod_{i=1}^n\mathds 1_{\{\tau_i^\delta=\infty\}}
-
\prod_{i=1}^n\mathds 1_{\{\tau_i=\infty\}}, 
\sum_{i=1}^n\mathds 1_{\{\tau_i^\delta=\infty\}}
-
\sum_{i=1}^n\mathds 1_{\{\tau_i=\infty\}}\Big)
\le
\sum_{i=1}^n
\mathds 1_{\{\tau_i<\infty,\ \tau_i^\delta=\infty\}}.
\end{equation}
For each \(i\), notice that the event \(\{\tau_i<\infty,\ \tau_i^\delta=\infty\}\)
can occur only if, after the time \(\tau_i\), the \(i\)-th shifted coordinate,
which is then equal to \(\delta\), survives forever. Since its drift is bounded
above by \(b+1\), the strong Markov property and comparison with a Brownian
motion with constant drift \(b+1\) provide
\(
\mathbb P_x[\tau_i<\infty,\ \tau_i^\delta=\infty]\le H_{b+1}(\delta)
\). Taking expectations in \eqref{ineq:cont} and plugging the last inequality induces the required \eqref{eq:uniform-continuity-value}.
\ep
\end{proof}

\section{Triviality in the small drift setting}

In this section, we show that both problems have zero value function when the drift parameter $b$ is sufficiently small. 
As a consequence, the main results of this paper  that justify the validity of Conjectures 1 and 2 will be proved outside of this small drift setting.

\begin{proposition}\label{prop:roughpro-intro}
The maps $U^n$ and $V^n$ vanish in the following situations:

\vspace{3mm}
\hspace{-2mm} {\rm (i)} $V^n=0$ on $E_n$ if and only if $b\le -\frac{1}{n}$;

\vspace{3mm}
\hspace{-2mm} {\rm (ii)} $U^n=0$ on $E_n$ if and only if $b\le -1$;
\end{proposition}

\begin{proof}
(i) By the lower bound on $V^n$ in Proposition \ref{prop:firststeps} (i), we see that $b>-\frac1n$ implies that $V^n>0$. In the alternative case $b\le -\frac1n$, we note that for an arbitrary $\Phi\in\mathcal A_n$, we have $nb+\mathbf{1}^{\!n}\cdot \phi_t\le 0$ and therefore $S_t:=\mathbf{1}^{\!n}\cdot X_t^{\Phi}\le \mathbf{1}^{\!n}\cdot x+\mathbf{1}^{\!n}\cdot W_t$, and  $\mathbb P_x[S_t>0,\ \forall t\ge 0]=0$. As $\{\min_{i\le n}\tau_i^\Phi=\infty\} \subseteq \{S_t>0,\ \forall t\ge 0\}$, it follows that $
\mathbb P_x\!\left[\min_{i\le n}\tau_i^\Phi=\infty\right]=0$.
Since $\Phi\in\mathcal A_n$ is arbitrary,  we conclude that $V^n=0$.

(ii) By Proposition \ref{prop:firststeps} (i), we see that $b\le -1$ implies that $H_{b+1}\equiv 0$, and hence, $U^n=0$. 
In the alternative case $b>-1$, observe that the constant control $\Phi=(1,0,\ldots, 0)\in\mathcal A_n$ induces the lower bound $
U^n(x)\ge
\mathbb P_x[x_1+(b+1)t+W_t^1>0,\ \forall t\ge 0]
=
H_{b+1}(x_1)>0$ for $x\in E_n$.
\ep
\end{proof}

\vspace{5mm}
 Next, we show that the trivial setting of zero value function in 
Proposition \ref{prop:roughpro-intro} can also be characterized by the HJB equation \eqref{eq:hjb} together with the obvious behavior at the boundary of the positive orthant:
\begin{equation}\label{boundarycond}
w(x)=\left\{\begin{array}{lcl}
             0 &\mbox{for}& V^n
             \\
             U^{n-|\mathbf{0}(x)|}(x^{-0}) &\mbox{for}& U^n,
             \end{array}
      \right.
~~~\text{for all}~x\in\partial E_n.
\end{equation}
We emphasize that the last boundary condition, only on $\partial E_n$, is not sufficient for the characterization of the value function for general values of the drift parameter $b$. Indeed $w\equiv 0$ is a solution to \eqref{eq:hjb} satisfying both boundary conditions in \eqref{boundarycond}, but, in view of Proposition \ref{prop:roughpro-intro} (ii), it is clearly different from the value functions $V^n$ and $U^n$ when  $b>-\frac1n$ and $b>-1$,  respectively.  

\begin{proposition}\label{prop:negativedrift}
Let $n\ge 2$.

\medskip

\noindent {\rm (i)} If $b\le -\frac1n$, then $V^n\equiv 0$ is the unique bounded classical solution to the HJB equation \eqref{eq:hjb} with boundary condition $w=0$ on $\partial E_n$. 

\medskip

\noindent {\rm (ii)} If $b\le -1$, then $U^n\equiv 0$ is the unique bounded classical solution to the HJB equation \eqref{eq:hjb} with boundary condition $w(x)=U^{n-|\mathbf{0}(x)|}(x^{-0})$ for all $x\in\partial E_n$.
\end{proposition}

\begin{proof}
Let $w\in C^2(E_n)\cap C(\mathbb R_+^n)$ be a bounded solution to the PDE \eqref{eq:hjb}. For an arbitrary strategy $\Phi\in\Ac_n$, we introduce the standard localization $\tau^\Phi_K
:=\tau^\Phi\wedge\sigma_K$,  where $\tau^\Phi:=\min_{1\le i\le n}\tau^\Phi_i
                          =\inf\{t>0:X^{\Phi}_t\not\in E_n\}$ and $\sigma_K
:=
\inf\{t>0:\max(X^\Phi_t)>K\}$.
By It\^o's formula, we obtain for all $T>0$:
\be
w(x)
&=&
\E_x\big[w\big(X^\Phi_{T\wedge\tau^\Phi_K}\big)\big]
-
\E_x\Big[\int_0^{T\wedge\tau^\Phi_K}
                \Big(\frac12\Delta w+(b\, \mathbf{1}^{\!n}+\Phi_t)\cdot\nabla w\Big)(X^\phi_t)\mathrm{d}t
    \Big] 
\nonumber \\
&\ge&
\E_x\big[w\big(X^\Phi_{T\wedge\tau^\Phi_K}\big)\big]
\underset{K\to\infty}{\longrightarrow}
\E_x\big[w\big(X^\Phi_{T\wedge\tau^\Phi}\big)\big]
=
\E_x\big[w\big(X^\Phi_{\tau^\Phi}\big)\mathds 1_{\{ \tau^\Phi\le T\}}
          +w\big(X^\Phi_{T}\big)\mathds 1_{\{ \tau^\Phi>T\}}
          \big]
\nonumber
\\ && \hspace{35mm}
\underset{T\to\infty}{\longrightarrow}
\E_x\big[w\big(X^\Phi_{\tau^\Phi}\big)\mathds 1_{\{ \tau^\Phi<\infty\}}
          +\limsup_{T\to\infty}w\big(X^\Phi_{T}\big)\mathds 1_{\{ \tau^\Phi=\infty\}}
          \big],
\label{verif0}
\ee
where the first inequality is due to the HJB equation satisfied by $w$, and the convergence is ensured by the boundedness of $w$ and $\lim_{K\to\infty}\sigma_K=\infty$. Notice that in both cases (i) and (ii), we have $b\le -\frac1n$, and therefore \(\hat S_t:=\1^{\!n}\cdot \hat X_t 
\le \1^{\!n}\cdot (x+W_t)\) as $\1^{\!n}\cdot (b\1^{\!n}+\hat \Phi_t)=nb+1\le 0$. Then,  \( \mathbb P_x[\tau^{\hat\Phi}=\infty] \le \mathbb P_x[\hat S_t>0,\, \forall t\ge 0] = 0\), and we deduce from \eqref{verif0} that
\be\label{verif1}
w(x)
&\ge&
\E_x\big[w\big(X^\Phi_{\tau^\Phi}\big)\mathds 1_{\{ \tau^\Phi<\infty\}}
          \big].
 \ee
To establish the equality, we consider the feedback control $\hat\Phi_t:={\mathds 1}_{\{ i=\hat{I}(\hat X_t)\}}$ with state process $\hat X_t=X^{\hat\Phi}_t$ solving the corresponding rank-dependent drift SDEs:
\begin{equation}\label{eq:optimalcontrol}
\mathrm{d} \hat{X}^i_t =\big(b+ {\mathds 1}_{\{ i=\hat{I}(\hat X_t)\}}\big) \mathrm{d} t+ \mathrm{d} W^i_t \mbox{ and } \hat I(x)\!:=\!\min\big\{1\!\le\! j \!\le\! n: \partial_{x_j}w(x)\!=\!\max (\nabla w(x))\big \};
~x\!\in\! E_n.
\end{equation} 
The well-posedness of this SDE is guaranteed by \cite[Theorem 4]{zvonkin1974transformation}, and we obtain by following  the same line of argument as in \eqref{verif1} that 
\be\label{verif2}
w(x)
=
\E_x\big[w\big(\hat X_{\tau^{\hat\Phi}}\big)
           \mathds 1_{\{ \tau^{\hat\Phi}<\infty\}}
          \big].
\ee
Now, if $w$ satisfies the boundary conditions $w|_{\partial E_n} =0$, then the last equality immediately implies that $w=0$ on $\R_+^n$. 
If $w(x)=U^{n-|\mathbf{0}(x)|}(x^{-0})$ for $x\in\partial E_n$, then we proceed by induction:
\begin{itemize}
\item For $n=1$, the claim holds as $U^1(x)=H_{b+1}(x)=0$; 
\item Assume that the claim holds for $1,\ldots, n-1$, then $w(x)=U^{n-|\mathbf{0}(x)|}(x^{-0})=0$ for $x\in\partial E_n$, and we again deduce from \eqref{verif2} that $w(x)=0$. 
 \ep
\end{itemize}
\end{proof}

\section{Main results}\label{sec:main}

From the preliminary analysis of the previous section, it appears clearly that the relevant
nontrivial regimes are different for the two objectives:
\be\label{setting:mainresult}
n\ge2,
~~
b>-\frac1n
~\mbox{for } V^n,
~~\mbox{and}~~ 
b>-1
~\mbox{for } U^n.
\ee
The distinction between these two thresholds is important for the survivor-count summarized in the large time behavior stated in \Cref{thm:HJB} below. Our main result is as follows.

\begin{theorem}\label{thm:conjecture12}
In the nontrivial regime setting of \eqref{setting:mainresult}:
\begin{itemize}
\item[{\rm (i)}] {\rm Conjectures 1} and {\rm 2} hold for $V^n$.
\item[{\rm (ii)}] {\rm Conjectures 1} and {\rm 2} do not hold for $U^n$.
\end{itemize}
\end{theorem}

We report the proof of this statement in Sections \ref{sec:conjectureU} and \ref{sec:conjectureV}.
The key ingredient is the derivation of the appropriate recursive boundary condition at infinity.
For $V^n$, the recursion removes a particle at infinity and reduces $V^n$  to $V^{n-1}$ if $b>-1/n$.  For $U^n$ in the range $b>-1$, the recursion also decreases the remaining survivor cap,  and we need for this reason to introduce the \emph{capped survivor-count} value functions: 
\[
U^{n,r}(x):=
\sup_{\Phi\in\Ac_n}
\E_x\Big[
r\wedge\sum_{i=1}^n \mathds 1_{\{\tau_i^\Phi=\infty\}}
\Big],
\qquad
x\in\R_+^m,
~\text{for all}~
r=0,\ldots,n,
\]
with the conventions $
U^{n,0}\equiv0,$ $U^{n,r}=U^{n,n}=U^n$ for $r\ge n$, $U^{0,r}\equiv0$. For $b>-1$,  let us define
\[
\kappa_n(b):=\max\{1\le k\le n: kb>-1\},
\qquad n\ge1.
\]
Then, as will be shown in Section~\ref{sec:largeparticles}, no admissible strategy can make more than $\kappa_n(b)$ particles  survive forever in dimension $n$, and therefore,
\(
U^n=U^{n,\kappa_n(b)}.
\) 

The following HJB characterization specifies the exact nature of boundary and growth conditions at infinity, where, for a subset $I\subset\{1,\ldots,n\}$, $y^{-I}\longrightarrow \infty$ means $\min_{j\in I^c}y_j\longrightarrow\infty$.

\begin{theorem}\label{thm:HJB}
Consider the nontrivial setting of \eqref{setting:mainresult}. Then:
\begin{itemize}
\item[{\rm (i)}] $V^n$ is the unique
$C^2(E_n)\cap C_b(\R_+^n)$ solution to the HJB equation \eqref{eq:hjb} with 
\be
&&
\text{boundary condition}~~~~~~w=0
~~\text{on}~~\partial E_n,
\label{bc0:V}
\\
&&
\text{joint trace at infinity}~~~~
\lim_{y\to(x^I,\infty)} w(y)
=
V^{|I|}(x^I),
~~\text{for all}~~
I\subset\{1,\ldots,n\}.
\label{bc:V}
\ee
In particular, 
\(
\lim_{x\to\infty}V^n(x)=1.
\)
Moreover, an optimal Markov feedback is given by
\[
\hat\Phi^i_t
:=
\mathds 1_{\{i=\hat I(\hat X_t)\}},
~~
\hat I(x)
:=
\min\big\{1\le j\le n:
\partial_{x_j}V^n(x)=\max(\nabla V^n(x))\big\},
\]
with state process solving the corresponding feedback SDE as in \eqref{eq:optimalcontrol}.

\item[{\rm (ii)}] For $1\!\le\! r\!\le\!\kappa_n(b)$, $U^{n,r}$ is the
unique $C^2(E_n) \cap  C_b(\R_+^n)$ solution to the HJB equation \eqref{eq:hjb} with
\be
&&
\text{boundary condition}~~~~w(x)=U^{n-|\mathbf 0(x)|,r}(x^{-0}),
~~x\in\partial E_n,
\label{bc0:Ur}
\\
&&
\text{joint trace at infinity}~
\lim_{y\to(x^I,\infty)} w(y)
=
(q\wedge r)+U^{|I|,(r-|I^c|)^+}(x^I),
~
I\!\subset\!\{1,\ldots,n\}.~~~~~
\label{bc:Ur}
\ee
In particular, 
\(
\lim_{x\to\infty}U^{n,r}(x)=r.
\) 
Moreover, 
an optimal Markov feedback is given  by 
\[
\hat\Psi^{r,i}_t
:=
\mathds 1_{\{i=\hat J_r(\hat X_t)\}},
~~
\hat J_r(x)
:=
\min\big\{1\le j\le n:
\partial_{x_j}U^{n,r}(x)=\max(\nabla U^{n,r}(x))\big\},
\]
with state process solving the corresponding feedback SDE as in \eqref{eq:optimalcontrol},  and by continuing recursively with the corresponding lower-dimensional feedback after a particle
hits zero.
\item[{\rm (iii)}] Consequently,  \(U^n=U^{n,\kappa_n(b)}\) satisfies
\be\label{bc:U}
U^n(x)=U^{n-|\mathbf 0(x)|}(x^{-0}),
~ x\in\partial E_n,
~\text{and}~
\lim_{y\to(x^I\!,\infty)} \!\!U^n(y)
=
|I^c|\!\wedge\!\kappa_n(b)
+
U^{|I|,(\kappa_n(b)-|I^c|)^+}\!(x^I).
\ee
In particular
\(
\lim_{x\to\infty}U^n(x)=\kappa_n(b).
\)
\end{itemize}
\end{theorem}

The proof of this result is reported in Sections \ref{sec:largeparticles} and \ref{sec:HJB}.

\section{Large time particle asymptotics in the nontrivial regime}
\label{sec:largeparticles}

This section justifies the recursive boundary conditions at infinity \eqref{bc:V} and \eqref{bc:Ur},
which is crucial for the PDE characterization of Theorem \ref{thm:HJB}. 
From the verification argument in \Cref{prop:negativedrift}, 
 it requires to understand the large-time behavior of the particles on the default event
in the nontrivial regime.
We give the precise statement in Lemma \ref{lem:preparation-remote} in the Appendix, 
and we  highlight here the main feature which drives our key-arguments:
\begin{eqnarray*}
&\lim_{t\to\infty}X^{\Phi,i}_t=\infty
~~\mbox{on}~\{\tau_i^\Phi=\infty\},
~\mbox{for all}~i=1,\ldots,n.
\end{eqnarray*}
We also notice that the same coupling argument as in Proposition \ref{prop:firststeps} shows that the capped
value functions $U^{n,r}$ are symmetric, componentwise nondecreasing, and uniformly
continuous on $\R_+^n$.

 \begin{proposition}\label{prop:recursive-trace-negative}
In the nontrivial regime setting of \eqref{setting:mainresult}, we have for all $x\!\in\!E_n$ and all $I\!\subset\!\{1,\ldots,n\}$:
\begin{itemize}
\item[\rm (i)] 
\(
\lim_{y\to(x^I,\infty)}V^n(y)=V^{|I|}(x^I).
\) 

\item[\rm (ii)] For all $1\le r\le\kappa_n(b)$,
\(
\lim_{y\to(x^I,\infty)}U^{n,r}(y)
=
(q\wedge r)+U^{|I|,(r-q)^+}(x^I),
\) with $q:=n-|I|$.
In particular,
\(
\lim_{x\to\infty}U^{n,r}(x)=r.
\) 

\item[\rm (iii)] No admissible strategy can do better than
$\kappa_n(b)$ survivals, consequently,
\(
U^n=U^{n,\kappa_n(b)}.
\)
\end{itemize}
\end{proposition}

\begin{proof}
By symmetry, it is enough to prove the one-coordinate traces when the coordinate sent to
infinity is the last one.

\medskip

\noindent
{\rm (i) The one-coordinate trace for \(V^n\).}
For an arbitrary
\(\Phi=(\phi^1,\ldots,\phi^n)\in\Ac_n\), note that  the control
\(
\widetilde\Phi:=(\phi^1,\ldots,\phi^{n-1})
\) 
 defined by restricting to the first $n-1$ coordinates is admissible, allowing the harmless
extra randomness coming from the \(n\)-th Brownian motion. Then, for \(x\in E_{n-1}\), we have:
\[
\E_{x,z}
\Big[
\prod_{i=1}^n\mathds 1_{\{\tau_i^\Phi=\infty\}}
\Big]
\le
\E_x
\Big[
\prod_{i=1}^{n-1}\mathds 1_{\{\tau_i^{\tilde\Phi}=\infty\}}
\Big]
\le
V^{n-1}(x),
~~\mbox{for all}~z>0.
\]
Taking the supremum over \(\Phi\in\Ac_n\), we obtain
\begin{equation}\label{growth:upperbound}
V^n(x,z)\le V^{n-1}(x),
~~\mbox{for all}~ z>0.
\end{equation}
For the reverse inequality, fix \(\varepsilon>0\), and choose
\(\Psi=(\psi^1,\ldots,\psi^{n-1})\in\Ac_{n-1}\), realized independently of \(W^n\), such that
\[
\E_x\Big[
\prod_{i=1}^{n-1}\mathds 1_{\{\tau_i^\Psi=\infty\}}
\Big]
\ge
V^{n-1}(x)-\varepsilon.
\]
For \(z>0\), define an \(n\)-dimensional control by
\[
\Phi^z_t
:=
(\Psi_t,0)\mathds 1_{[0,\sqrt z]}(t)
+
\frac1n\1^n\mathds 1_{(\sqrt z,\infty)}(t).
\]
On \([0,\sqrt z]\), the first \(n-1\) coordinates follow \(\Psi\), while
\[
X_t^{\Phi^z,n}=z+bt+W_t^n,
\qquad 0\le t\le\sqrt z.
\]
Fix \(R>0\), and set
\[
F_i^{z,R}
:=
\{\tau_i^{\Phi^z}\ge\sqrt z,\ X_{\sqrt z}^{\Phi^z,i}\ge R\},
\qquad i=1,\ldots,n,
\qquad
F^{z,R}:=\bigcap_{i=1}^nF_i^{z,R}.
\]
By \eqref{eq:fixed-R-preparation} of Lemma \ref{lem:preparation-remote}, applied
coordinate-wise to the first \(n-1\) coordinates under \(\Psi\), we have
\(
\prod_{i=1}^{n-1}\mathds 1_{F_i^{z,R}}
\xrightarrow[]{z\to\infty}
\prod_{i=1}^{n-1}\mathds 1_{\{\tau_i^\Psi=\infty\}},
\) a.s. which provides by dominated convergence:
\[
\E_{x,z}
\Big[
\prod_{i=1}^{n-1}\mathds 1_{F_i^{z,R}}
\Big]
\xrightarrow[]{z\to\infty}
\E_x
\Big[
\prod_{i=1}^{n-1}\mathds 1_{\{\tau_i^\Psi=\infty\}}
\Big].
\]
Moreover, by \eqref{eq:remote-coordinate-estimate-general} of Lemma
\ref{lem:preparation-remote}, applied with constant drift \(b\), we have
\(
\P[F_n^{z,R}]\longrightarrow1.
\)
Since \(F_n^{z,R}\) is independent of the first \(n-1\) coordinates, it follows that
\[
\E_{x,z}[\mathds 1_{F^{z,R}}]
\longrightarrow
\E_x
\Big[
\prod_{i=1}^{n-1}\mathds 1_{\{\tau_i^\Psi=\infty\}}
\Big].
\]
On \(F^{z,R}\), all particles are alive at time \(\sqrt z\) and are at least at level \(R\).
After time \(\sqrt z\), each coordinate receives drift \(b+\frac1n>0\). Hence, by the Markov
property,
\[
V^n(x,z)
\ge
\E_{(x,z)}[\mathds 1_{F^{z,R}}]\,H_{b+\frac1n}(R)^n.
\]
Letting first \(z\to\infty\), then \(R\to\infty\), and finally
\(\varepsilon\downarrow0\), we obtain
\[
\liminf_{z\to\infty}V^n(x,z)\ge V^{n-1}(x).
\]
Together with the upper bound \eqref{growth:upperbound}, this proves
\(
\lim_{z\to\infty}V^n(x,z)=V^{n-1}(x).
\)

\medskip

\noindent
{\rm (ii)} We prove first the one-coordinate trace. Fix $b>-1$ and
$1\le r\le \kappa_n(b)$. For $\Phi\in\Ac_n$, set
\[
N_n^\Phi:=\sum_{i=1}^n\mathds 1_{\{\tau_i^\Phi=\infty\}},
~~\mbox{and}~~
N_{n-1}^\Phi:=\sum_{i=1}^{n-1}\mathds 1_{\{\tau_i^\Phi=\infty\}}.
\]
Taking expectations in the inequality $r\wedge N_n^\Phi
\le
1+\big((r-1)\wedge N_{n-1}^\Phi\big)$, and maximizing over $\Phi\in\Ac_n$,
\begin{equation}\label{growth:upperboundU}
U^{n,r}(x,z)\le 1+U^{n-1,r-1}(x),
~~\mbox{for all}~~x\in E_{n-1},~z>0.
\end{equation}
We now prove the reverse inequality. Fix $\varepsilon>0$ and choose
$\Psi=(\psi^1,\ldots,\psi^{n-1})\in\Ac_{n-1}$, realized independently of $W^n$, such that
\[
\E_x\Big[
(r-1)\wedge
\sum_{i=1}^{n-1}\mathds 1_{\{\tau_i^\Psi=\infty\}}
\Big]
\ge
U^{n-1,r-1}(x)-\varepsilon.
\]
Fix $R>0$, and define
\b*
F_i^{z,R}
&:=&
\Big\{
\tau_i^\Psi>\sqrt z,\quad X_{\sqrt z}^{\Psi,i}\ge R
\Big\},
~~ i=1,\ldots,n-1,
~~
q_z^R:=
(r-1)\wedge\sum_{i=1}^{n-1}\mathds 1_{F_i^{z,R}},
\\
F_n^{z,R}
&:=&
\Big\{
z+bt+W_t^n>0,\ 0\le t\le\sqrt z,\quad
z+b\sqrt z+W_{\sqrt z}^n\ge R
\Big\}.
\e*
Choose, in a measurable way, a subset
$I_z^R\subset\{1,\ldots,n-1\}$ consisting of $q_z^R$ indices among those for which
$F_i^{z,R}$ occurs. For instance, take the smallest such indices.

Define an admissible $n$-dimensional strategy $\Phi^{z,R}$ as:
\[
\phi_t^{z,R,i}
=
\psi_t^i\mathds 1_{[0,\sqrt z]}(t)
+
\frac{\mathds 1_{\{i\in I_z^R\}}\mathds 1_{F_n^{z,R}}}{1+q_z^R}
\mathds 1_{(\sqrt z,\infty)}(t),
~ i=1,\ldots,n-1,
~\text{and}~
\phi_t^{z,R,n}
=
\frac{\mathds 1_{F_n^{z,R}}}{1+q_z^R}
\mathds 1_{(\sqrt z,\infty)}(t).
\]
\begin{itemize}
\item On $[0,\sqrt z]$, the first $n-1$ coordinates follow $\Psi$, and the $n$-th coordinate
receives no resource;
\item On $(\sqrt z,\infty)$, on the event $F_n^{z,R}$, the unit
budget is divided equally among the coordinates in $I_z^R\cup\{n\}$, and outside
$F_n^{z,R}$, allocate the budget arbitrarily. 
\end{itemize}
By \eqref{eq:fixed-R-preparation}, applied coordinate-wise to the first $n-1$
coordinates under $\Psi$, we see that $
q_z^R
\xrightarrow[z\to\infty]{}
(r-1)\wedge
\sum_{i=1}^{n-1}\mathds 1_{\{\tau_i^\Psi=\infty\}},$ a.s. Moreover, it follows from \eqref{eq:remote-coordinate-estimate-general}, applied with
constant drift $b$, that $\P[F_n^{z,R}]\longrightarrow1$. Since $F_n^{z,R}$ is independent of $q_z^R$, dominated convergence gives
\[
\E_{(x,z)}
\left[
\mathds 1_{F_n^{z,R}}(1+q_z^R)
\right]
\longrightarrow
1+
\E_x\Big[
(r-1)\wedge
\sum_{i=1}^{n-1}\mathds 1_{\{\tau_i^\Psi=\infty\}}
\Big].
\]
On $F_n^{z,R}$, the selected particles are alive, and are all at least at level $R$
at time $\sqrt z$. The number of selected particles is $1+q_z^R\le r$, and each
selected particle receives after time $\sqrt z$ a drift at least $b+\frac1r>0$, because $r\le\kappa_n(b)$. Therefore, conditionally on $\Fc_{\sqrt z}$, the
expected capped number of selected survivors is bounded from below by $(1+q_z^R)H_{b+\frac1r}(R)$, and consequently,
\[
U^{n,r}(x,z)
\ge
H_{b+\frac1r}(R)
\E_{(x,z)}
\Big[(1+q_z^R)
        \mathds 1_{F_n^{z,R}}
\Big].
\]
Letting first $z\to\infty$, then $R\to\infty$, and finally
$\varepsilon\downarrow0$, we obtain $\liminf_{z\to\infty}U^{n,r}(x,z)
\ge
1+U^{n-1,r-1}(x)$, which together with the upper bound \eqref{growth:upperboundU} shows that
\[
\lim_{z\to\infty}U^{n,r}(x,z)
=
1+U^{n-1,r-1}(x).
\]
We now derive the joint trace in {\rm (ii)}. Let $I\subset\{1,\ldots,n\}$, and set $
J:=I^c$, $q:=|J|=n-|I|$, and denote $N_I^\Phi:=\sum_{i\in I}\mathds 1_{\{\tau_i^\Phi=\infty\}}$ and $N_J^\Phi:=\sum_{j\in J}\mathds 1_{\{\tau_j^\Phi=\infty\}}$ for all $\Phi\in\Ac_n$. Since $N_J^\Phi\le q$, we have the pathwise inequality
\(
r\wedge(N_I^\Phi+N_J^\Phi)
\le
(q\wedge r)+\big((r-q)^+\wedge N_I^\Phi\big).
\)
Taking expectations and optimizing gives
\[
U^{n,r}(x^I,z^J)
\le
(q\wedge r)+U^{|I|,(r-q)^+}(x^I)
~~\text{for all}~~
x^I\in E_{|I|}.
\]
For the reverse inequality, send the coordinates in $J$ to infinity one after another. Each
application of the one-coordinate trace contributes one survivor and decreases the remaining cap
by one, until the remaining cap reaches zero. The condition needed to apply the one-coordinate
trace is preserved during this recursion: if, after $a$ coordinates have been sent to infinity, the
remaining cap $r-a$ is positive, then $(r-a)b>-1$ and $r-a\le n-a$.
Hence $r-a\le\kappa_{n-a}(b)$. Therefore, the successive one-coordinate traces yield
\[
\lim_{z^J\to\infty}U^{n,r}(x^I,z^J)
=
(q\wedge r)+U^{|I|,(r-q)^+}(x^I),
\]
first as an iterated limit, and then a genuine joint limit, by component-wise monotonicity of $U^{n,r}$.

It remains to allow the finite coordinates to approach the boundary. Let $(y^m)_{m\ge 1}\subset E_n$ be s.t. $(y^m)^I\to x^I\in\R_+^{|I|}$ and $(y^m)^J\to\infty$.
The upper bound above and the continuity of the lower-dimensional capped value functions give
\[
\limsup_{m\to\infty}U^{n,r}(y^m)
\le
(q\wedge r)+U^{|I|,(r-q)^+}(x^I).
\]
For the lower bound, set
\[
K:=\{i\in I:x_i>0\},
~~
L:=I\setminus K,
~\text{and define}~
x_\eta^K:=(x_i-\eta)_{i\in K}\in E_{|K|},
~\text{for some}~
\eta>0.
\]
For $m$ large enough, $(y^m)^K\ge x_\eta^K$ component-wise. Hence, by component-wise monotonicity and the
finite-boundary identity for the coordinates in $L$,
\[
U^{n,r}(y^m)
\ge
U^{n,r}(x_\eta^K,0^L,(y^m)^J)
=
U^{n-|L|,r}(x_\eta^K,(y^m)^J).
\]
Applying the already proved fixed-coordinate joint trace in this reduced system, with the above
conventions for capped values, gives
\[
\liminf_{m\to\infty}U^{n,r}(y^m)
\ge
(q\wedge r)+U^{|K|,(r-q)^+}(x_\eta^K)
\xrightarrow[\eta\downarrow0]{}
(q\wedge r)+U^{|I|,(r-q)^+}(x^I),
\]
by continuity and the finite-boundary identity of the
lower-dimensional capped value function.
This proves the joint trace in {\rm (ii)}. Finally, taking $I=\emptyset$ gives $\lim_{x\to\infty}U^{n,r}(x)=r$.

\medskip
\noindent
{\rm (iii)} We only focus on the non-trivial situation $k:=\kappa_n(b)<n$. Then $kb+1>0$ and $(k+1)b+1\le0$. We claim that no $k+1$ particles can survive simultaneously. Indeed, for any $I\subset\{1,\ldots,n\}$ with $|I|=k+1$, and for an arbitrary $\Phi\in\Ac_n$, the process $S_t^I:=\sum_{i\in I}X_t^{\Phi,i}$ satisfies
\[
S_t^I
=
S_0^I
+
\int_0^t
\Big((k+1)b+\sum_{i\in I}\phi_s^i\Big)\dd s
+
\sum_{i\in I}W_t^i
\le
S_0^I+\sum_{i\in I}W_t^i,
\qquad t\ge0,
\]
because $\sum_{i\in I}\phi_s^i\le1$ and $(k+1)b+1\le0$. If all particles in
$I$ survive forever, then $S_t^I>0$ for all $t\ge0$, and
\(
\P_x[\tau_i^\Phi=\infty,\ i\in I]
\le
\P_x\left[
S_0^I+\sum_{i\in I}W_t^i>0,\ \forall t\ge0
\right]
=0,
\)
since $\sum_{i\in I}W^i$ is a one-dimensional Brownian motion up to a deterministic time
change. Taking the union over the finitely many subsets $I$ of cardinality $k+1$, we obtain $\sum_{i=1}^n\mathds 1_{\{\tau_i^\Phi=\infty\}}
\le k$, a.s.
By the arbirariness of $\Phi\in\Ac_n$, and recalling that $k=\kappa_n(b)$, this implies that
\[
U^n(x)
=
\sup_{\Phi\in\Ac_n}
\E_x\Big[
\sum_{i=1}^n\mathds 1_{\{\tau_i^\Phi=\infty\}}
\Big]
=
\sup_{\Phi\in\Ac_n}
\E_x\Big[
k\wedge
\sum_{i=1}^n\mathds 1_{\{\tau_i^\Phi=\infty\}}
\Big]
=
U^{n,\kappa_n(b)}(x).
\]
Together with {\rm (ii)}, this provides the corresponding capped trace formulas for $U^n$
when needed.
\ep
\end{proof}

\section{HJB characterization}\label{sec:proof-wellposed}\label{sec:HJB}

This section is devoted to the proof of Theorem \ref{thm:HJB}. The local regularity is obtained
by standard arguments in interior elliptic regularity theory. We recall the local function spaces
used below.

Let  $O\subset\mathbb R^n$ be an open set, and $p\in(1,\infty)$. 
The Sobolev space
$W^{2,p}(O)$ consists of all functions $u\in L^p(O)$,
whose weak derivatives up to order
two belong to $L^p(O)$, endowed with the norm
\[
\|u\|_{W^{2,p}(O)}
:=
\sum_{|\beta|\le 2}\|D^\beta u\|_{L^p(O)}.
\]
We denote by
\(
W^{2,p}_{\mathrm{loc}}(E_n):=\cap_K W^{2,p}(K),
\)
where the intersection is over all compact subsets $K\Subset E_n$.
Similarly, for $\alpha\in(0,1)$ and an open set $O\subset\mathbb R^n$, the H\"older space
$C^{2,\alpha}(O)$ consists of all maps $u\in C^2(O)$,
endowed with the norm
\[
\|u\|_{C^{2,\alpha}(O)}
:=
\sum_{|\beta|\le 2}\|D^\beta u\|_{L^\infty(O)}
+
\sum_{|\beta|=2}[D^\beta u]_{C^{0,\alpha}(O)}
<\infty,
\]
where
\[
[w]_{C^{0,\alpha}(O)}
:=
\sup_{\substack{x,y\in O\\x\neq y}}
\frac{|w(x)-w(y)|}{|x-y|^\alpha}.
\]
We denote
\(
C^{2,\alpha}_{\mathrm{loc}}(E_n):=\cap_K C^{2,\alpha}(K),
\)
where the intersection is over all compact subsets $K\Subset E_n$.

\vspace{5mm}
\noindent {\bf Proof of Theorem \ref{thm:HJB}.}
We proceed by induction on $n$. 
For $n=1$, we have
\[
V^1=U^{1,1}=U^1=H_{b+1}\in C^2(E_1)\cap C(\mathbb R_+),
\]
which is the unique bounded classical solution with the prescribed boundary behavior. Assume now
that the statement holds in all dimensions strictly smaller than $n$.

\vspace{3mm}
\noindent {\it Step 1. The value functions solve the HJB equation and satisfy the boundary conditions.}
Assume first that $b>-\frac1n$. By Proposition \ref{prop:firststeps}, $V^n$ is continuous on
$\R_+^n$. By standard stochastic control theory, $V^n$ satisfies the dynamic programming
principle, and is therefore a bounded viscosity solution to the HJB equation \eqref{eq:hjb} on
$E_n$. Its finite boundary condition is obvious from the definition, and its trace at infinity
is given by Proposition \ref{prop:recursive-trace-negative}.

Assume next that $b>-1$, and fix $1\le r\le \kappa_n(b)$. The same coupling argument as in
Proposition \ref{prop:firststeps} shows that $U^{n,r}$ is continuous on $\R_+^n$; indeed it is
symmetric, componentwise nondecreasing, and satisfies the same uniform continuity estimate as
$U^n$. The dynamic programming principle then implies that $U^{n,r}$ is a bounded viscosity
solution to \eqref{eq:hjb} on $E_n$. The finite boundary condition is
\(
U^{n,r}(x)=U^{n-|\mathbf 0(x)|,r}(x^{-0}),\qquad x\in\partial E_n,
\)
and the trace at infinity is given by Proposition \ref{prop:recursive-trace-negative}. Finally,
taking $r=\kappa_n(b)$ and using
\(
U^n=U^{n,\kappa_n(b)}
\)
gives the corresponding assertions for the original value function $U^n$.

\vspace{3mm}
\noindent {\it Step 2. Interior regularity.}
We show that the value functions obtained in Step 1 are in
$C^{2,\alpha}_{\mathrm{loc}}(E_n)\cap C(\mathbb R_+^n)$ for some $\alpha\in(0,1)$.
Let $W$ denote either $V^n$ or one of the capped value functions $U^{n,r}$,
$1\le r\le\kappa_n(b)$. The argument is identical in all cases, since it only uses boundedness,
local Lipschitz continuity, and the HJB equation on $E_n$.

Fix compact sets $K\Subset K'\Subset E_n$,  
where $K \Subset K'$ means that $K$ is a compact set and is contained in the interior of $K'$.   
Since $W$ is Lipschitz on $\mathbb R^n_+$,
Rademacher's theorem implies that $W$ is differentiable almost everywhere on $K'$ and
$\nabla W\in L^\infty(K')$. Since $W$ is a viscosity solution to \eqref{eq:hjb}, the uniform
ellipticity of the equation implies that it is also a distributional solution; see
Caffarelli--Cabr\'e \cite[Proposition~2.9]{CaffarelliCabre}. Therefore,
\[
\Delta W
=
f
:=
-2b\,\mathbf 1^{\!n}\cdot\nabla W
-
2\max(\nabla W)
\in L^\infty(K').
\]
By the interior Calder\'on--Zygmund estimate for Poisson's equation, see
Gilbarg--Trudinger \cite[Theorem~9.11]{GT}, we have, for all $p\in(1,\infty)$,
\[
\|W\|_{W^{2,p}(K)}
\le
C_{K,K',p}
\Big(
\|W\|_{L^p(K')}
+
\|f\|_{L^p(K')}
\Big),
\]
for some constant $C_{K,K',p}$. Hence $W\in W^{2,p}(K)$. Choosing $p>n$, the Sobolev
embedding theorem, see Gilbarg--Trudinger \cite[Theorem~7.26]{GT}, gives
\(
W^{2,p}(K)\hookrightarrow C^{1,\alpha}(K),\) with \(\alpha:=1-\frac np\in(0,1)\).
Thus $W\in C^{1,\alpha}(K)$. Since the map $q\mapsto\max(q)$ is Lipschitz,
we have $f\in C^{0,\alpha}(K)$. The interior Schauder estimates for Poisson's equation,
see Gilbarg--Trudinger \cite[Theorems~6.2 and~6.6]{GT}, then imply that $W\in C^{2,\alpha}(K)$. As $K\Subset E_n$ is arbitrary, this proves
\[
V^n\in C^{2,\alpha}_{\mathrm{loc}}(E_n)\cap C(\mathbb R_+^n),
\qquad
U^{n,r}\in C^{2,\alpha}_{\mathrm{loc}}(E_n)\cap C(\mathbb R_+^n),
~~\text{for every}~~
1\le r\le\kappa_n(b).
\]
\noindent {\it Step 3. A verification identity.}
We shall use the following standard consequence of It\^o's formula. Let
$w\in C^2(E_n)\cap C(\mathbb R_+^n)$ be a bounded classical solution to \eqref{eq:hjb},
and let $\Phi\in\Ac_n$ be arbitrary. Set
\[
\tau^\Phi:=\min_{1\le i\le n}\tau_i^\Phi,
\qquad
\sigma_K:=\inf\{t>0:\max(X_t^\Phi)>K\}.
\]
Applying It\^o's formula to $w(X^\Phi)$ on $[0,T\wedge\tau^\Phi\wedge\sigma_K]$ gives
\[
w(x)
=
\E_x\big[w(X^\Phi_{T\wedge\tau^\Phi\wedge\sigma_K})\big]
-
\E_x\Big[
\int_0^{T\wedge\tau^\Phi\wedge\sigma_K}
\left(
\frac12\Delta w+(b\mathbf 1^{\!n}+\Phi_t)\cdot\nabla w
\right)(X_t^\Phi)\dd t
\Big].
\]
By the HJB equation, the integrand is nonpositive. Letting first $K\to\infty$ and then
$T\to\infty$, and using the boundedness of $w$, yields
\begin{equation}\label{eq:verification-ineq-general}
w(x)
\ge
\E_x\left[
w(X^\Phi_{\tau^\Phi})\mathds 1_{\{\tau^\Phi<\infty\}}
+
\limsup_{T\to\infty}w(X_T^\Phi)\mathds 1_{\{\tau^\Phi=\infty\}}
\right].
\end{equation}
Let $\hat\tau:=\min_{1\le i\le n}\hat\tau_i,$ $\hat\tau_i:=\inf\{t\!>\!0:\hat X_t^i\!=\!0\}$ where $\hat X=X^{\hat\Phi}$ is the state corresponding to 
\[
\hat\Phi^i_t
:=
\mathds 1_{\{i=\hat I(\hat X_t)\}},
\qquad
\hat I(x)
:=
\min\left\{
1\le j\le n:
\partial_{x_j}w(x)=\max(\nabla w(x))
\right\}.
\]
The well-posedness of the SDE defining $\hat X$ follows as in
\eqref{eq:optimalcontrol}.
Then, the equality holds in the HJB equation along $\hat\Phi$, and we obtain by the same localization argument:
\begin{equation}\label{eq:verification-eq-general}
w(x)
=
\E_x\left[
w(\hat X_{\hat\tau})\mathds 1_{\{\hat\tau<\infty\}}
+
\lim_{T\to\infty}w(\hat X_T)\mathds 1_{\{\hat\tau=\infty\}}
\right],
\end{equation}
whenever the limit on $\{\hat\tau=\infty\}$ exists. 

\vspace{3mm}
\noindent {\it Step 4. Uniqueness and verification for $V^n$.}
Let $v\in C^2(E_n)\cap C(\mathbb R_+^n)$ be a bounded classical solution to
\eqref{eq:hjb} satisfying the boundary conditions in Theorem \ref{thm:HJB} {\rm (i)}.
For an arbitrary $\Phi\in\Ac_n$, \eqref{eq:verification-ineq-general} and the boundary condition
$v=0$ on $\partial E_n$ give
\[
v(x)
\ge
\E_x\left[
\limsup_{T\to\infty}v(X_T^\Phi)\mathds 1_{\{\tau^\Phi=\infty\}}
\right].
\]
On $\{\tau^\Phi=\infty\}$, Lemma \ref{lem:preparation-remote} implies, coordinate by coordinate, that $X_T^{\Phi,i}\xrightarrow[T\to\infty]{}\infty$, a.s. $i=1,\ldots,n$.
Then, by the joint trace condition \eqref{bc:V}, $\lim_{T\to\infty}v(X_T^\Phi)=1$ on $\{\tau^\Phi=\infty\}$, and therefore $
v(x)\ge \P_x[\tau^\Phi=\infty]$.
Since $\Phi\in\Ac_n$ is arbitrary,
\[
v(x)\ge V^n(x).
\]
Conversely, let $\hat\Phi$ be the feedback control associated with $v$ as in Step 3. Then
\eqref{eq:verification-eq-general}, the finite boundary condition, Lemma
\ref{lem:preparation-remote}, and \eqref{bc:V} give
\[
v(x)
=
\P_x[\hat\tau=\infty]
\le
V^n(x).
\]
Hence $v=V^n$. This also proves the optimality of the feedback control in
Theorem \ref{thm:HJB} {\rm (i)}.

\vspace{3mm}
\noindent {\it Step 5. Uniqueness and verification for $U^{n,r}$.}
Assume that $b>-1$, fix $1\le r\le\kappa_n(b)$, and let
$u\in C^2(E_n)\cap C(\mathbb R_+^n)$ be a bounded classical solution to \eqref{eq:hjb}
satisfying the finite boundary condition
\[
u(x)=U^{n-|\mathbf 0(x)|,r}(x^{-0}),
\qquad x\in\partial E_n,
\]
and the joint trace condition \eqref{bc:Ur}.  Let $\Phi\in\Ac_n$ be arbitrary, and denote
$N_n^\Phi:=\sum_{i=1}^n\mathds 1_{\{\tau_i^\Phi=\infty\}}$.
By \eqref{eq:verification-ineq-general}, the finite boundary condition, Lemma
\ref{lem:preparation-remote}, and the joint trace \eqref{bc:Ur} with $I=\emptyset$, we obtain
\begin{equation}\label{eq:verification-U-upper}
u(x)
\ge
\E_x\left[
U^{n-|\mathbf 0(X^\Phi_{\tau^\Phi})|,r}\big((X^\Phi_{\tau^\Phi})^{-0}\big)
\mathds 1_{\{\tau^\Phi<\infty\}}
+
r\mathds 1_{\{\tau^\Phi=\infty\}}
\right].
\end{equation}
On $\{\tau^\Phi<\infty\}$, the particles in $\mathbf 0(X^\Phi_{\tau^\Phi})$ have already failed.
The restriction of the continuation control to the remaining coordinates is admissible in the
lower-dimensional system. Therefore, by the definition of the lower-dimensional capped value
function,
\[
\E_x\left[
r\wedge N_n^\Phi
\ \big|\ \mathcal F_{\tau^\Phi}
\right]
\le
U^{n-|\mathbf 0(X^\Phi_{\tau^\Phi})|,r}\big((X^\Phi_{\tau^\Phi})^{-0}\big)
\quad\mbox{on }\{\tau^\Phi<\infty\}.
\]
On $\{\tau^\Phi=\infty\}$, all $n$ particles survive, and hence
\(
r\wedge N_n^\Phi=r.
\)
Combining these observations with \eqref{eq:verification-U-upper} gives
\(
u(x)
\ge
\E_x\big[r\wedge N_n^\Phi\big].
\)
Taking the supremum over $\Phi\in\Ac_n$ yields
\[
u(x)\ge U^{n,r}(x).
\]
We now prove the reverse inequality. Let $\hat\Psi$ be the feedback associated with $u$,
\[
\hat\Psi^{i}_t
:=
\mathds 1_{\{i=\hat J(\hat X_t)\}},
\qquad
\hat J(x)
:=
\min\left\{
1\le j\le n:
\partial_{x_j}u(x)=\max(\nabla u(x))
\right\},
\]
up to the first hitting time $\hat\tau$. Then \eqref{eq:verification-eq-general}, the finite
boundary condition, Lemma \ref{lem:preparation-remote}, and \eqref{bc:Ur} give
\begin{equation}\label{eq:verification-U-equality}
u(x)
=
\E_x\left[
U^{n-|\mathbf 0(\hat X_{\hat\tau})|,r}\big(\hat X_{\hat\tau}^{-0}\big)
\mathds 1_{\{\hat\tau<\infty\}}
+
r\mathds 1_{\{\hat\tau=\infty\}}
\right].
\end{equation}
On $\{\hat\tau<\infty\}$, we continue recursively with an optimal lower-dimensional feedback
for
\[
U^{n-|\mathbf 0(\hat X_{\hat\tau})|,r}
\big(\hat X_{\hat\tau}^{-0}\big),
\]
with the convention $U^{m,r}=U^{m,m}=U^m$ when $r\ge m$ and $U^{0,r}\equiv0$.
This lower-dimensional optimal feedback is available by the induction hypothesis. The
$n$-dimensional strategy obtained by concatenating $\hat\Psi$ before $\hat\tau$ with these
lower-dimensional optimal feedbacks after $\hat\tau$ is admissible, and its expected capped
survivor count is exactly the right-hand side of \eqref{eq:verification-U-equality}. Hence
\(
U^{n,r}(x)\ge u(x).
\)
Together with the opposite inequality, this proves
\(
u=U^{n,r}.
\)
The same argument proves the optimality of the recursive feedback described in
Theorem \ref{thm:HJB} {\rm (ii)}.

Finally, taking $r=\kappa_n(b)$ and using Proposition \ref{prop:recursive-trace-negative} gives
\(
U^n=U^{n,\kappa_n(b)},
\)
and therefore the HJB characterization and verification statement for the original value function
$U^n$. This completes the proof of Theorem \ref{thm:HJB}.
\ep

\section{Conjectures fail for $U^n$}
\label{sec:conjectureU}

We proceed in two steps. First, we prove the failure in dimension two. This argument is local near the corner and only uses the positivity of $b+1$, hence it works for every $b>-1$. Second, we lift the two-dimensional failure to higher dimensions. In the range where the survivor cap is smaller than the dimension, the lifting has to be performed with the capped values $U^{m,r}$.

In the remaining of this section, we denote by a slight abuse of notation
\[u(x_1, x_2) = u = u(r, \alpha),\] 
for both Cartesian and polar coordinates, when the context is clear. 
The gradient $\nabla u$ and the Laplacian $\Delta u$ are with respect to the Cartesian coordinates.

\subsection{The two-dimensional counterexample}

Throughout this subsection, assume $b>-1$ and write
\[
u:=U^{2,\kappa_2(b)}=U^2=U^{2,2}\mathds 1_{\{b>-\frac12\}}+U^{2,1}\mathds 1_{\{b\le-\frac12\}}.
\]
\begin{theorem}\label{thm:counterexample-2d}
There exists a non-empty open set $O\!\subset\!\{(x_1,x_2)\!\in\! E_2\!:x_1\!<\!x_2\}$ such that $\partial_{x_2}u>\partial_{x_1}u$ on $O$.
\end{theorem}

\proof
We argue by contradiction, assuming that
\begin{equation}\label{eq:ass-ineq-b}
\partial_{x_1}u(x_1,x_2)\ge \partial_{x_2}u(x_1,x_2),
~~\text{for all}~~ (x_1,x_2)\in D:=\{(x_1,x_2)\in E_2:\ x_1<x_2\}.
\end{equation}
Then \(u\) satisfies the linear equation
\begin{equation}\label{eq:u2-linear}
\mathcal L u=0
~~\text{on}~~D,
~~\text{with}~~
\mathcal L
:=
\frac12\Delta+(b+1)\partial_{x_1}+b\partial_{x_2}.
\end{equation}
The boundary of \(D\) decomposes into $\partial D=\Gamma_D\cup\Gamma_N$, where
$\Gamma_D:=\{(0,x_2):x_2>0\}$ and $\Gamma_N:=\{(z,z):z>0\}$. By Theorem \ref{thm:HJB} and the symmetry of $u$, we have the boundary conditions
\begin{equation}\label{eq:u2-Neu}
u(0,.)
=
U^1
=
H_{b+1}
~\text{on}~\Gamma_D,
~~\text{and}~~
\partial_n u=0
~\text{on }\Gamma_N,
\end{equation}
where \(n\) denotes the inward normal to \(D\) along \(\Gamma_N\), proportional to
\((-1,1)\). We now analyze the mixed boundary problem near the corner \(0\). We introduce polar coordinates in the wedge \(D\):
\[
x_1=r\cos\left(\frac{\pi}{4}+\alpha\right),
\qquad
x_2=r\sin\left(\frac{\pi}{4}+\alpha\right),
\qquad
\alpha\in\Big(0,\frac{\pi}{4}\Big).
\]
The Neumann condition \eqref{eq:u2-Neu} becomes
$\partial_\alpha u(r,0)=0$,  
whereas the Dirichlet condition reads
$u(r,\frac{\pi}{4})=1-\mathrm e^{-2\mu r}
=
2\mu r-2\mu^2r^2+O(r^3)$ with
$\mu:=b+1$.  Define
\[
v:=u-P_1-P_2,
~~\text{where}~~
P_1(x_1,x_2):=2\mu(x_1+x_2)
~~\text{and}~~
P_2(x_1,x_2):=-2\mu^2(x_1^2+x_2^2)
\]
are the first two elementary profiles satisfying $P_1(0,x_2)+P_2(0,x_2)
=
2\mu x_2-2\mu^2x_2^2$, and $\partial_{x_1}P_i=\partial_{x_2}P_i$
on $\{x_1=x_2\}$, $i=1,2$. Straightforward computation provides:
\begin{equation}\label{eq:v-corner-eq}
\mathcal L v=2\mu+O(r)
~~\text{near }0,
~~\text{and}~~
v(0,x_2)=O(x_2^3)
~~\text{as}~~
x_2\downarrow0,
~~
\partial_n v=0
~~\text{on}~~\Gamma_N.
\end{equation}
We next observe that
\begin{itemize}
\item in the sector $0<\alpha<\frac{\pi}{4}$, the artificial boundary \(\alpha=0\) corresponds to the diagonal \(x_1=x_2\),
\item the symmetry of \(u\) gives the homogeneous Neumann condition $\partial_\alpha v(r,0)=0$,
\item the side \(\alpha=\pi/4\) corresponds to \(x_1=0\), and by the boundary conditions in \eqref{eq:v-corner-eq} its Dirichlet datum is \(O(r^3)\). 
\end{itemize}
As $\mathcal L v=2\mu+O(r)$ by \eqref{eq:v-corner-eq}, we are in the context of Lemma \ref{lem:mixed-corner-expansion} below, which yields:
\begin{equation}\label{eq:v-expansion}
v(r,\alpha)
=
c_b\,r^2\log r\,\cos(2\alpha)
+
R(r,\alpha),
~~\text{with}~~
|\nabla R(r,\alpha)|\le C r
~~\text{for}~~
(r,\alpha)\in (0,r_0)\times\Big[0,\frac{\pi}{4}\Big],
\end{equation}
with $c_b:=\frac{4\mu}{\pi}>0$. By the last estimate on $\nabla R$, we get:
\begin{eqnarray}
(\partial_{x_2}-\partial_{x_1})u 
&=&
(\partial_{x_2}-\partial_{x_1})\left\{c_br^2\log r\,\cos(2\alpha)+R(r,\alpha)\right\}
\nonumber\\
&=&
2c_b(x_2-x_1)\left(-\log r+\frac{x_1x_2}{r^2}\right)
-4\mu^2(x_2-x_1)
+O(r),
\label{eq:diff-sign}
\end{eqnarray}
as \(r\downarrow0\), uniformly on closed subsectors of \(D\).  Now fix any angle $\theta_0\in\left(\frac{\pi}{4},\frac{\pi}{2}\right)$ and consider points
$(x_1,x_2)=r(\cos\theta_0,\sin\theta_0)$. Note that the leading term in \eqref{eq:diff-sign} is
$2c_b r(\sin\theta_0-\cos\theta_0)(-\log r)>0$ as \(r\downarrow0\), because
\(c_b>0\), and dominates all \(O(r)\) terms.  Therefore, along this ray,
\[
\partial_{x_2}u(r\cos\theta_0,r\sin\theta_0)
>
\partial_{x_1}u(r\cos\theta_0,r\sin\theta_0)
~~\text{for all sufficiently small}~~
r>0,
\]
which is the required contradiction to \eqref{eq:ass-ineq-b}. Since \(u\in C^2(E_2)\), the map $(\partial_{x_2}-\partial_{x_1})u$ is continuous, and we may then a nonempty
open subset \(O\subset D\) fulfilling the required strict inequality.
\ep

\begin{lemma}\label{lem:mixed-corner-expansion}
Let $f,g,v$ be maps on $S_\rho:=(0,\rho)\times(0,\frac{\pi}{4})$, $\rho>0$, satisfying for some \(a\in\mathbb R^2,\mu>0\):
\begin{eqnarray*}
&&\frac12\Delta v+a\cdot\nabla v
=2\mu+f,
~\text{on}~S_\rho,
~\text{with boundary conditions}~
\partial_\alpha v\big|_{\alpha=0}=0,
~\text{and}~
v\big|_{\alpha=\frac\pi{4}}=g(r),
\\
&&|f(r,\alpha)|\le Cr,
~~
|g(r)|+r|g'(r)|+r^2|g''(r)|\le Cr^3,
~~\text{as}~~r\downarrow0.
\end{eqnarray*}
Then there exist \(\rho_0\in(0,\rho)\), \(C>0\), and a remainder \(R\) such that
\[
v(r,\alpha)
=
\frac{4\mu}{\pi}r^2\log r\cos(2\alpha)+R(r,\alpha),
~\mbox{and}~
|\nabla R(r,\alpha)|\le Cr,
~(r,\alpha)\in S_{\rho_0},
\] where \(\nabla\) denotes the Cartesian gradient.
\end{lemma}

\begin{proof}
This is the standard mixed Dirichlet--Neumann corner expansion for a uniformly
elliptic operator with smooth coefficients in a planar sector,    see Dauge
\cite[Chapter~4, Sections~13--14]{Dauge88}. We only identify explicitly the
coefficient of the first resonant term, since this is the only term used in the
proof of Theorem \ref{thm:counterexample-2d}.  We recall the
identification of the leading coefficient, since it is the only coefficient
needed below.  For the principal part \(\frac12\Delta\), the angular eigenfunctions associated
with the homogeneous mixed boundary conditions
\[
\Theta'(0)=0,
\qquad
\Theta(\omega)=0
\]
are
\[
e_k(\alpha)=\cos(\lambda_k\alpha),
\qquad
\lambda_k=\frac{\pi/2+k\pi}{\omega}=2+4k,
\qquad k=0,1,\ldots .
\]
Thus the first singular exponent is \(\lambda_0=2\), with eigenfunction
\(e_0(\alpha)=\cos(2\alpha)\). Since the right-hand side has constant leading
term \(2\mu\), its projection onto \(e_0\) is
\[
a_0
=
\frac{\int_0^\omega 2\mu\cos(2\alpha)\,\dd\alpha}
     {\int_0^\omega \cos^2(2\alpha)\,\dd\alpha}
=
\frac{\mu}{\pi/8}
=
\frac{8\mu}{\pi}.
\]
Moreover,
\[
\frac12\Delta\big(r^2\log r\cos(2\alpha)\big)
=
2\cos(2\alpha).
\]
Therefore the resonant coefficient multiplying
\(r^2\log r\cos(2\alpha)\) is \(a_0/2=4\mu/\pi\).

The lower-order term \(a_1\partial_{x_1}+a_2\partial_{x_2}\) does not change
this coefficient: applied to the resonant profile it produces only terms of
order \(r|\log r|\), which generate corrections of order \(r^3|\log r|\) after
solving the second-order equation. These corrections have Cartesian gradient
\(O(r)\). Similarly, the boundary datum \(g=O(r^3)\), together with its first two
derivative bounds, contributes only to the remainder with gradient \(O(r)\).
The remaining homogeneous modes are \(r^{2+4k}\cos((2+4k)\alpha)\), \(k\ge0\),
and their gradients are \(O(r)\) or smaller near the corner. This gives the
claimed expansion and the estimate on \(R\).\ep 
\end{proof}

\subsection{A capped dimension-lifting argument}

The next lemma is the capped version of the dimension-lifting argument. It is the form compatible
with the recursive trace at infinity from Proposition \ref{prop:recursive-trace-negative}.

\begin{lemma}\label{lem:capped-lifting}
Let $n\ge2$ and let $1\le r\le\kappa_n(b)\wedge[\kappa_{n+1}(b)-1]$. Then:
\[
w_R:=U^{n+1,r+1}(.,R)-1
\;\xrightarrow[R\to\infty]{}\;
U^{n,r}
~~\mbox{in}~~
C^2_{\rm loc}(E_n).
\]
\end{lemma}

\begin{proof}
We prove the convergence on an arbitrary compact set $K\Subset E_n$. Choose bounded open sets
$O,O'\subset E_n$ such that $K\Subset O\Subset O'\Subset E_n$, and set $Q:=O\times(-\rho,\rho)$ and $Q':=O'\times(-2\rho,2\rho)$ for some fixed $\rho>0$. For $R>2\rho$, consider the translated functions
\[
u_R(x,s):=U^{n+1,r+1}(x,R+s)-1,
\qquad (x,s)\in Q'.
\]
Note that $(x,R+s)$ belongs to $E_{n+1}$ for every $(x,s)\in Q'$. Moreover,
by the HJB equation for $U^{n+1,r+1}$, the function $u_R$ solves
\[
\frac12(\Delta_xu_R+\partial_{ss}u_R)
+
b(\1^n\cdot\nabla_xu_R+\partial_su_R)
+
\max(\partial_{x_1}u_R,\ldots,\partial_{x_n}u_R,\partial_su_R)
=0
\qquad\mbox{in }Q'.
\]
The equation is translation-invariant in the last coordinate, and $1\le u_R\le r$ on $Q'$, because $0\le U^{n+1,r+1}\le r+1$. Applying the same interior regularity argument as in Step 2 of the proof of Theorem \ref{thm:HJB} on the fixed pair $Q\Subset Q'$, we obtain constants
$\alpha\in(0,1)$ and $C_K>0$, independent of $R$, such that $\|u_R\|_{C^{2,\alpha}(Q)}\le C_K,$ for $R>2\rho$. Restricting to the slice $\{s=0\}$ gives $\|w_R\|_{C^{2,\alpha}(K)}\le C_K$, $R>2\rho$. By the compact embedding $C^{2,\alpha}(K)\hookrightarrow C^2(K)$, this shows that the family $(w_R)_{R>2\rho}$ is relatively compact in $C^2(K)$.

It remains to identify the possible limit. Let $R_\ell\to\infty$ be a sequence such that that $w_{R_\ell}\to w$ in $C^2(K)$. By Proposition
\ref{prop:recursive-trace-negative}, applied in dimension $n+1$ with cap $r+1$, we have $w_{R_\ell}=U^{n,r}$, pointwise on $E_n$, implying that $w=U^{n,r}$ on $K$, independently of the choice of the converging sequence. This shows that $w_R\longrightarrow U^{n,r}$ in $C^2(K)$, and the required result follows from the arbitrariness of $K\Subset E_n$.\ep 
\end{proof}

\subsection{Failure beyond the two-dimensional setting}

We first introduce the following lemma. 

\begin{lemma}\label{lem:cap-one-failure}
Let $n\ge2$ and $b>-1$. Then, $\partial_{x_2}U^{n,1}>\partial_{x_1}U^{n,1}$ on some a nonempty open subset of $\{x\in E_n:x_1<\min_{j\ge2}x_j\}$.
\end{lemma}

\begin{proof}
Denote $D_n:=(\partial_{x_2}-\partial_{x_1})U^{n,1}$, $n\ge2$, and let us prove the  slightly stronger claim:
\begin{equation}\label{strengthened}
\text{for all}~\delta>0, 
~\text{there exists}~
O_{n,\delta}\subset
\Big\{
x\in E_n:\ 0<x_1<\delta\min_{2\le j\le n}x_j
\Big\}
~\text{such that}~
D_n>0~\text{on}~O_{n,\delta},
\end{equation}
where $O_{n,\delta}$ may be chosen to be open, by the regularity of $D_n$. 
We shall proceed by induction on $n$, and we need for this to establish the following property:
\begin{equation}\label{eq:finite-face-stability-cap-one}
U^{n+1,1}(\cdot,\eta)\xrightarrow[\eta\downarrow0]{} 
U^{n,1}
~\mbox{in}~C^2(K),
~~\text{for all}~
n\ge1
~\text{and all compact set}~
 K\Subset E_n.
\end{equation}
Indeed, by ignoring the last coordinate,
\(
U^{n+1,1}(x,\eta)\ge U^{n,1}(x),
\) for $x\in\R_+^m$, $\eta>0.$
Conversely, for any admissible strategy in dimension $n+1$, the restriction to the first $n$
coordinates is admissible in dimension $n$, and the last coordinate can survive with probability
at most $H_{b+1}(\eta)$. Hence, $0\le U^{n+1,1}(x,\eta)-U^{n,1}(x)\le H_\mu(\eta)$, implying that $U^{n+1,1}(\cdot,\eta)\to U^{n,1}$ locally uniformly on $E_n$ as $\eta\downarrow0$.
The convergence of derivatives follows by the same compactness argument as in
Lemma \ref{lem:capped-lifting}, applied now near the flat finite face
$\{x_{n+1}=0\}$. More precisely, the finite-boundary condition gives
\(
U^{n+1,1}(\cdot,0)=U^{n,1}
\) in $E_n$. 
On each half-cylinder $K'\times(0,\eta_0)$ with $K\Subset K'\Subset E_n$, boundary Schauder
estimates for the uniformly elliptic HJB equation give uniform $C^{2,\alpha}$ estimates in the
tangential variables. Hence, the family $U^{n+1,1}(\cdot,\eta)$ is relatively compact in $C^2(K)$ as
$\eta\downarrow0$, and the locally uniform convergence above identifies the only possible limit
as $U^{n,1}$. This proves \eqref{eq:finite-face-stability-cap-one}.

We now proceed with our induction argument starting from the claim \eqref{strengthened} for $n=2$. The proof of
Theorem \ref{thm:counterexample-2d} applies  to $U^{2,1}$. 
Indeed, the HJB equation is the same, the symmetry condition on the diagonal is the same, and the finite-boundary condition is $U^{2,1}(0,.)=U^1=H_{b+1}$.
Moreover, the local corner argument is local in every smaller sector $\left\{(x_1,x_2)\in E_2:\ 0<x_1<\delta x_2\right\}$, $\delta>0$.
If $D_2\le0$ throughout such a sector near the origin, then $U^{2,1}$ satisfies there the same linear equation used in the proof of Theorem \ref{thm:counterexample-2d}. The mixed Dirichlet--Neumann corner expansion then gives, along a ray
$x_1=r\cos\theta_0$, $x_2=r\sin\theta_0$ with $\frac{\cos\theta_0}{\sin\theta_0}<\delta$, the asymptotic inequality $D_2(r\cos\theta_0,r\sin\theta_0)>0$
for all sufficiently small $r>0$, a contradiction. Hence, for every $\delta>0$,
there are points with $0<x_1<\delta x_2$ at which $D_2>0$.

Assume now that the strengthened claim \eqref{strengthened} holds in dimension $n\ge2$. Fix $\delta'>0$ and let $\delta>\frac{2\delta'}{1-\delta'}$.
By the induction hypothesis, $D_n>0$ on some nonempty open set $O_{n,\delta'}\subset
\big\{
x\in E_n:\ 0<x_1<\delta'\min_{2\le j\le n}x_j
\big\}
$.
Choose $x^*\in O_{n,\delta'}$ and $\gamma>0$ such that
\(
D_n(x^*)\ge4\gamma.
\)
Shrinking the open set if necessary, we may choose a compact set
$K\Subset O_{n,\delta'}$ with $x^*$ in its interior and such that
\(
D_n(x)\ge3\gamma
\) for all $x\in K$.  
By \eqref{eq:finite-face-stability-cap-one}, there exists $\eta_0>0$ such that, for all
$0<\eta<\eta_0$ and $x\in K$,
\[
\left|
\partial_{x_i}U^{n+1,1}(x,\eta)-\partial_{x_i}U^{n,1}(x)
\right|<\gamma,
\qquad i=1,2.
\]
Consequently,
\[
D_{n+1}(x,\eta)
=
(\partial_{x_2}
-
\partial_{x_1})U^{n+1,1}(x,\eta)
\ge
D_n(x)-2\gamma
\ge
\gamma,
~~\text{for all}~x\in K
~\text{and}~0<\eta<\eta_0.
\]
Since the strengthened assertion in dimension $n$ gives failure points with the first coordinate arbitrarily small relative to the remaining coordinates, we choose $x^*\in K$ so that, in addition, $x_1^*<\frac12\delta'\eta_0$. Then we may choose $\eta$ satisfying $
\frac{x_1^*}{\delta'}<\eta<\eta_0$ and $\eta<\min_{2\le j\le n}x_j^*$.
Set
\(
\hat x:=(x^*,\eta)\in E_{n+1}.
\)
Then
\(
D_{n+1}(\hat x)>0.
\)
Moreover, since $x_1^*<\delta'\eta$, and since
$x^*\in O_{n,\delta'}$, it follows that $x_1^*<\delta'\min_{2\le j\le n}x_j^*$, and therefore
\[
\hat x_1
<
\delta'\min_{2\le j\le n+1}\hat x_j
<
\delta\min_{2\le j\le n+1}\hat x_j .
\]
Thus $\hat x$ belongs to the desired conical region in dimension $n+1$ and satisfies
$D_{n+1}(\hat x)>0$. By continuity of $\nabla U^{n+1,1}$ on $E_{n+1}$, the same strict inequality
holds on a nonempty open neighborhood of $\hat x$, and shrinking this neighborhood if necessary
keeps it inside $\big\{
x\in E_{n+1}:\ 0<x_1<\delta\min_{2\le j\le n+1}x_j
\big\}$.
\ep 
\end{proof}

\vspace{5mm}
\noindent {\bf Proof of \Cref{thm:conjecture12} (ii)} We disprove both conjectures in two steps.

\medskip

\noindent{\em Step 1}. We first show that Conjecture 2 fails for $U^n$, namely that $(\partial_{x_2}-\partial_{x_1})U^n>0$ on some subset of $\{x\in E_n:x_1<\min_{j\ge2}x_j\}$, which will be chosen to be open by the continuity of $\nabla U^n$.
By Proposition \ref{prop:recursive-trace-negative},
\(
U^n=U^{n,\kappa_n(b)}.
\) We distinguish two cases.

\medskip
\noindent
{\it Case 1: Suppose \(k:=\kappa_n(b)<n\).}
Then $m_0:=n-k+1\ge2$. By Lemma \ref{lem:cap-one-failure}, Conjecture 2 fails for $U^{m_0,1}$. Hence there exist
$x^{(0)}\in E_{m_0}$ and $\varepsilon>0$ such that
\[
x^{(0)}_1<\min_{2\le j\le m_0}x^{(0)}_j,
\qquad
\partial_{x_2}U^{m_0,1}(x^{(0)})
-
\partial_{x_1}U^{m_0,1}(x^{(0)})
\ge2\varepsilon .
\]
We now lift this inequality successively. For
\(
\ell=0,\ldots,k-2,
\)
set
\(
m_\ell:=m_0+\ell
\) and $
r_\ell:=1+\ell.$  Since $r_\ell+1\le k$ and $kb>-1$, we have
\(
r_\ell+1\le\kappa_{m_\ell+1}(b).
\)
Therefore Lemma \ref{lem:capped-lifting} applies and gives
\[
U^{m_\ell+1,r_\ell+1}(\cdot,R)-1
\longrightarrow
U^{m_\ell,r_\ell}
\qquad\mbox{in }C^2_{\rm loc}(E_{m_\ell}),
\qquad R\to\infty.
\]
Starting from $x^{(0)}$, choose inductively $R_{\ell+1}>0$ large enough and define
\(
x^{(\ell+1)}:=(x^{(\ell)},R_{\ell+1})\in E_{m_{\ell+1}}
\) for $\ell=0,\ldots,k-2$, 
so that
\[
\partial_{x_2}U^{m_{\ell+1},r_{\ell+1}}(x^{(\ell+1)})
-
\partial_{x_1}U^{m_{\ell+1},r_{\ell+1}}(x^{(\ell+1)})
\ge\varepsilon .
\]
We also choose each $R_{\ell+1}$ larger than $\max(x^{(\ell)})$. Thus the first coordinate remains
the unique smallest coordinate throughout the induction.
After $k-1$ lifting steps, we obtain a point
\(
\hat x:=x^{(k-1)}\in E_{m_0+k-1}=E_n
\)
such that
\(
\hat x_1<\min_{2\le j\le n}\hat x_j
\)
and
\(
\partial_{x_2}U^{n,k}(\hat x)
-
\partial_{x_1}U^{n,k}(\hat x)
\ge\varepsilon .
\)
Since $U^n=U^{n,k}$, this proves $\partial_{x_2}U^n(\hat x)>\partial_{x_1}U^n(\hat x)$, $\hat x_1=\min(\hat x)$.

\medskip
\noindent
{\it Case 2: \(\kappa_n(b)=n\).}
If $n=2$, the result is exactly Theorem \ref{thm:counterexample-2d}. For $n\ge3$, apply Theorem \ref{thm:counterexample-2d} in the full two-dimensional cap to obtain the existence of $x^{(0)}\in E_2$ and $\varepsilon>0$ such that
\[
x^{(0)}_1<x^{(0)}_2,
\qquad
\partial_{x_2}U^{2,2}(x^{(0)})
-
\partial_{x_1}U^{2,2}(x^{(0)})
\ge2\varepsilon .
\]
We now lift from $(2,2)$ to $(n,n)$. For
\(
\ell=0,\ldots,n-3,
\)
set
\(
m_\ell:=2+\ell,
\) $r_\ell:=2+\ell.$
As $b>-\frac1n$, we have $(m_\ell+1)b>-1$, $\ell=0,\ldots,n-3$, and therefore
\(
r_\ell+1=m_\ell+1\le\kappa_{m_\ell+1}(b).
\)
Lemma \ref{lem:capped-lifting} gives
\[
U^{m_\ell+1,r_\ell+1}(\cdot,R)-1
\longrightarrow
U^{m_\ell,r_\ell}
\qquad\mbox{in }C^2_{\rm loc}(E_{m_\ell}),
\qquad R\to\infty.
\]
As in the previous case, we choose $R_{\ell+1}$ inductively large enough and define
\(
x^{(\ell+1)}:=(x^{(\ell)},R_{\ell+1})
\)
so that
\(
\partial_{x_2}U^{m_{\ell+1},r_{\ell+1}}(x^{(\ell+1)})
-
\partial_{x_1}U^{m_{\ell+1},r_{\ell+1}}(x^{(\ell+1)})
\ge\varepsilon .
\)
Choosing also $R_{\ell+1}>\max(x^{(\ell)})$ at every step, the first coordinate remains the unique
smallest coordinate.

After $n-2$ dimensions have been reached, we obtain $\hat x\in E_n$ such that
\(
\hat x_1<\min_{2\le j\le n}\hat x_j
\)
and
\(
\partial_{x_2}U^{n,n}(\hat x)
-
\partial_{x_1}U^{n,n}(\hat x)
\ge\varepsilon .
\)
Since $k=n$ and hence $U^n=U^{n,n}$, this proves the desired strict inequality in the full-cap
case.

\medskip
Combining the two cases, there exists $\hat x\in E_n$ such that
\[
\hat x_1=\min(\hat x),
\qquad
\partial_{x_2}U^n(\hat x)>\partial_{x_1}U^n(\hat x).
\]
Finally,  by the continuity of 
\(
x\longmapsto
\partial_{x_2}U^n(x)-\partial_{x_1}U^n(x)
\),  the strict inequality persists on a sufficiently small neighborhood of
$\hat x$. Shrinking this neighborhood if necessary, we may keep it inside
\(
\{x\in E_n:\ x_1<\min_{j\ge2}x_j\}.
\)

\medskip
\noindent {\em Step 2}. We next show that Conjecture 1 fails for $U^n$. By the first step of the current proof, there exists a nonempty open set $O\subset E_n$ such
that, for every $y\in O$, the first coordinate is the unique smallest coordinate and
\(
\partial_{x_1}U^n(y)<\max(\nabla U^n(y)).
\)
Let $x\in O$ and $r>0$ such that $\overline B(x,r)\subset O$, consider the controlled process $X=(X^1,\ldots,X^n)$ under the push-the-laggard strategy
$\overline\Phi$, started from $x$, and define
\[
\theta:=\inf\{t\ge0:X_t\notin B(x,r)\}.
\]
As $U^n\!\in\! C^2(E_n)\cap C(\R_+^n)$ satisfies the HJB equation \eqref{eq:hjb}, and the push-the-laggard strategy allocates the whole budget to the first
smallest coordinate, we have \(
\frac12\Delta U^n(y)
+
b\1^n\cdot\nabla U^n(y)
+
\sum_{j=1}^n\overline\phi^j(y)\partial_{x_j}U^n(y)
<0$ for all $y\in O$. By continuity and compactness of $\overline B(x,r)$, 
\[
\frac12\Delta U^n(y)
+
[b\1^n+\overline\phi(y)]\cdot\nabla U^n(y)
\le -\varepsilon,
~~ y\in\overline B(x,r),
~~\mbox{for some}~~
\varepsilon>0.
\]
Applying It\^o's formula on $[0,\theta\wedge t]$, and taking expectations, we see that
\[
\E_x[U^n(X_{\theta\wedge t})]
=
U^n(x)
+
\E_x\Big[\int_0^{\theta\wedge t}
\Big(
\frac12\Delta U^n
+
(b\1^n+\overline\Phi_s)\cdot\nabla U^n
\Big)(X_s)\dd s \Big]
\le
U^n(x)-\varepsilon\E_x[\theta\wedge t].
\]
Suppose by contradiction that $\overline\Phi$ were optimal for $U^n$ at $x$, namely
\(
U^n(x)=
\E_x\big[\sum_{j=1}^n\mathds 1_{\{\tau_j^{\overline\Phi}=\infty\}}\big]
=
\E_x\big[
\E\big[
\sum_{j=1}^n\mathds 1_{\{\tau_j^{\overline\Phi}=\infty\}}
\big| X_{\theta\wedge t}
\big]
\big]
\le
\E_x[U^n(X_{\theta\wedge t})]
$ by the strong Markov property, which contradicts the last inequality $\theta>0$ almost surely. 
Therefore, the push-the-laggard strategy is not
optimal for $U^n$.
\ep

\section{Conjectures hold  for $V^n$}\label{sec:conjectureV}

\noindent{\bf Proof of Theorem \ref{thm:conjecture12} (i)} We prove the validity of both conjectures in two steps.

\medskip
\noindent {\em Step 1}. To prove Conjecture 2, it is enough to establish the following pairwise comparison:
\[
\partial_{x_i}V^n(x)\ge \partial_{x_j}V^n(x)
~~\text{for all}~~i\neq j
~\text{and}~
x\in E_n
 ~\text{with}~
 x_i<x_j.
\]
By symmetry, it is enough to prove the statement for the pair
\((i,j)=(1,2)\). We shall in fact prove the following stronger finite-difference
statement:
\begin{equation}\label{eq:finite-diff-comparison}
V^n(x_1+h,x_2,x_3,\dots,x_n)
\ge
V^n(x_1,x_2+h,x_3,\dots,x_n),
~\text{for all}~
h>0~\text{and}~x\in E_n~\mbox{with}~x_1<x_2.
\end{equation}
We prove \eqref{eq:finite-diff-comparison} by induction on the dimension.
For \(n=1\) there is nothing to prove. Assume that the finite-difference
comparison has been proved in all dimensions strictly smaller than \(n\), and
let us prove it in dimension \(n\). Define
$
D:=\{x\in E_n:\ x_1<x_2\}$, and define for fixed \(h>0\):
\[
u(x):=V^n(x_1+h,x_2,x_3,\dots,x_n),
~~
\widetilde u(x):=V^n(x_1,x_2+h,x_3,\dots,x_n),~~ w_h(x):=u(x)-\widetilde u(x),
~~
x\in D.
\]
Since \(V^n\in C^2(E_n)\) solves \eqref{eq:hjb}, both \(u\) and
\(\widetilde u\) solve the same HJB equation on \(D\). Then
\begin{equation}\label{eq:wh-raw}
\frac12\Delta w_h
+
b\sum_{i=1}^n\partial_{x_i}w_h
+
\bigl(\max(\nabla u)-\max(\nabla\widetilde u)\bigr)
=0
~~\text{on}~~D.
\end{equation}
For each \(x\in D\), by convexity of \(q\mapsto \max(q)\), there exists
\(\lambda(x)=(\lambda_1(x),\ldots,\lambda_n(x))\in[0,1]^n\) with
\(\sum_{k=1}^n\lambda_k(x)=1\) such that $\max(\nabla u(x))-\max(\nabla\widetilde u(x))
=
\sum_{k=1}^n\lambda_k(x)\partial_{x_k}w_h(x)$.
Consequently,
 \(w_h\) satisfies the linear uniformly elliptic equation with bounded coefficients:
\[
Lw_h=0
~\text{on}~D,
~~\text{where}~~
L\phi
:=
\frac12\Delta\phi+\sum_{k=1}^n c_k(x)\partial_{x_k}\phi,
~~
c_k(x):=b+\lambda_k(x)\in[b,b+1].
\]
We next verify that \(w_h\) is nonnegative on the boundary of \(D\), including
the boundary at infinity. 
\begin{enumerate}
\item[{\rm (1.a)}] On the finite boundary, the verification is direct:
\begin{itemize} 
\item if \(x_1=x_2=z>0\), then by symmetry of \(V^n\) in the first two
coordinates,
\[
w_h(z,z,x^{-\{1,2\}})
=
V^n(z+h,z,x^{-\{1,2\}})
-
V^n(z,z+h,x^{\{1,2\}})
=0.
\]
\item if \(x_1=0\), then $
w_h(0,x^{-1})
=
V^n(h,x^{-1})
-
V^n(0,x_2+h,x^{-\{1,2\}})
=
V^n(h,x^{-1})
\ge0.$
\item Finally, if \(x_k=0\) for
some \(k\ge3\), then $w_h(x)=0$.
Hence
\begin{equation}\label{eq:wh-finite-boundary}
w_h\ge0
~~\text{on}~~\partial D.
\end{equation}
\end{itemize}
\item[{\rm (1.b)}] We next examine the behavior at infinity. By Theorem \ref{thm:HJB}, \(V^n\)
has the recursive boundary trace at infinity. Thus, \(V^n\), and hence \(w_h\),
extends continuously to the componentwise compactification
\([0,\infty]^n\), with the convention that sending one coordinate to
\(\infty\) removes that coordinate from the all-survive value function.  Namely,  we claim that:
\begin{equation}\label{eq:wh-infinity}
\liminf_{\substack{x\in D\\ |x|\to\infty}} w_h(x)\ge0.
\end{equation}
Let \(x^m\in D\) be any sequence escaping every compact subset of \(D\). Up to a
subsequence, \(x^m\) converges componentwise in \([0,\infty]^n\). If the limit
belongs to the finite boundary of \(D\), then the limiting value of \(w_h\) is
nonnegative by \eqref{eq:wh-finite-boundary}. Otherwise at least one coordinate
tends to \(\infty\). We distinguish the following cases.
\begin{itemize}
\item If \(x_2^m\to\infty\) while \(x_1^m\) stays finite, then the second coordinate
is removed in both terms in the limit, and it follows from the componentwise monotonicity of \(V^{n-1}\) that
\[
\lim_{m\to\infty}w_h(x^m)
=
V^{n-1}(x_1+h,x^{-\{1,2\}})
-
V^{n-1}(x_1,x^{-\{1,2\}})\ge0,
\]
\item If \(x_1^m\to\infty\), then,
because \(x_1^m<x_2^m\), also \(x_2^m\to\infty\), and the limiting value of \(w_h\) is \(0\).  
\item It remains to consider the case where \(x_1^m\) and \(x_2^m\) stay finite, while
at least one coordinate among \(x_3^m,\dots,x_n^m\) tends to \(\infty\). Removing
all coordinates which tend to \(\infty\), the limiting value of \(w_h\) is either
zero because a finite coordinate of the reduced vector lies on the absorbing
boundary, or else it is exactly a lower-dimensional finite difference of the
form
\[
V^k(y_1+h,y_2,y^{-\{1,2\}})
-
V^k(y_1,y_2+h,y^{-\{1,2\}}),
\qquad y_1<y_2,
\]
for some \(k<n\). By the induction hypothesis, this quantity is nonnegative.
\end{itemize}
\end{enumerate}
We now conclude by applying the weak maximum principle on bounded truncations,  see Gilbarg--Trudinger \cite[Theorem~3.1]{GT}.  We apply this result on the bounded domains
\(
D_R:=D\cap B_R,
\) for $R>0$,  where $B_R:=\{x\in E_n:\,  |x|<R\}$.
Let \(\eta>0\). By the boundary condition \eqref{eq:wh-infinity}, 
\[
w_h(x)\ge -\eta,~ 
 x\in D,~ |x|\ge R_\eta,  \mbox{ for some } R_\eta>0.
\]
Fix \(R>R_\eta\). On the finite part of \(\partial D_R\), we have \(w_h\ge-\eta\) by
\eqref{eq:wh-finite-boundary}; on the artificial boundary \(\partial B_R\cap D\), we have
\(w_h\ge-\eta\) by the choice of \(R_\eta\). Hence
 the
weak maximum principle gives
\(
w_h\ge - \eta
\) in $D_R$.   
Letting first \(R\to\infty\) and then \(\eta\downarrow0\), we obtain
\(
w_h\ge0\) in $D$.

\medskip
\noindent
{\em Step 2}. We now show that Conjecture 1 holds for $V^n$. Let $X=(X^1,\ldots,X^n)$ be the controlled process corresponding to the push-the-laggard strategy $\overline{\Phi}$,    starting from $x\in E_n$.  Define
\[
\tau:=\min_{1\le i\le n}\tau_i,
\qquad
\sigma_K:=\inf\{t\ge 0:\ \max(X_t)\ge K\},
\qquad K>0.
\]
Applying It\^o's formula to $V^n(X_s)$ on $[0,\tau\wedge\sigma_K\wedge t]$ yields
\begin{align*}
V^n(X_{\tau\wedge\sigma_K\wedge t})
&=V^n(x)+\int_0^{\tau\wedge\sigma_K\wedge t}
\Big(\frac12\Delta V^n+(b\mathbf 1^{\!n}+\overline{\Phi}_s)\cdot \nabla V^n\Big)(X_s)\,\dd s+\int_0^{\tau\wedge\sigma_K\wedge t}\nabla V^n(X_s)\cdot\dd W_s \\
&= V^n(x) +\int_0^{\tau\wedge\sigma_K\wedge t}\nabla V^n(X_s)\cdot\dd W_s,  
\end{align*}
as $V^n$ solves \eqref{eq:hjb}. Therefore,  taking expectations in both sides yields $V^n(x)= \mathbb E_x\big[V^n(X_{\tau\wedge\sigma_K\wedge t})\big]$. Letting $K\longrightarrow \infty$ first, and then $t\to\infty$,  we deduce from  the dominated convergence  theorem that 
\[
V^n(x)= \mathbb E_x\big[V^n(X_{\tau})\,\mathds 1_{\{\tau<\infty\}}+\lim_{t\to\infty}V^n(X_t)\,\mathds 1_{\{\tau=\infty\}}\big]= \mathbb E_x\big[\lim_{t\to\infty}V^n(X_t)\,\mathds 1_{\{\tau=\infty\}}\big].
\]
On $\{\tau=\infty\}$,  one has $\{\tau_i=\infty\}$ for $i=1,\ldots, n$,   and thus 
$$\limsup_{t\to\infty} X^i_t =\lim_{t\to\infty} X^i_t=\infty,
~~\text{and therefore}~~
\lim_{t\to\infty}V^n(X_t)=1.$$
Hence, $V^n(x)= \mathbb E_x\big[\mathds 1_{\{\tau=\infty\}}\big]$, which proves the optimality of $\overline\Phi$ for $V^n$. 
\qed 

\appendix

\section{Appendix}

\begin{lemma}\label{lem:preparation-remote}
Let $B$ be a Brownian motion, and $\tau_Y:=\inf\{t\ge0:\ Y_t\le0\}$ where $Y_t:=y+\int_0^t\beta_s\,\dd s+B_t$, $ t\ge 0$, for some $y>0$ and some bounded progressively measurable process $\beta$.
Then, for every fixed $R>0$,
\begin{eqnarray}
&\lim_{t\to\infty}Y_t=\infty
~~\mbox{on }\{\tau_Y=\infty\},
~~\lim\limits_{t\to\infty}
\mathds 1_{\{\tau_Y>t,\ Y_t\ge R\}}
= \mathds 1_{\{\tau_Y=\infty\}},
~\mbox{a.s.} &
\label{eq:fixed-R-preparation}
\\
&\mbox{and}~
\lim\limits_{y\to\infty}
\P\left[\tau_Y>\sqrt{y},~Y_{\sqrt{y}}\ge R
\right]
= 1
~~\mbox{if the process $\beta$ is constant.}&
\label{eq:remote-coordinate-estimate-general}
\end{eqnarray}
\end{lemma}

\begin{proof}
{\rm (i)} We first show the first part in \eqref{eq:fixed-R-preparation}.
Let $\theta$ be an arbitrary stopping time, and consider the event set $F:=\{\theta<\tau_Y,Y_\theta\le K\}\in\Fc_\theta$, for some $K>0$. Since the process $\beta$ is bounded, we have $\inf_{0\le s\le 1}Y_{\theta+s}\le Y_\theta+\|\beta\|_\infty+\inf_{0\le s\le 1}(B_{\theta+s}-B_\theta)$, implying with $\bar K:=K+\|\beta\|_\infty$ that
$$
F\cap\{\inf_{0\le s\le 1}(B_{\theta+s}-B_\theta)\le-\bar K\}
\subset
F\cap\{\inf_{0\le s\le 1}Y_{\theta+s}\le K\}
\subset
F\cap\{\tau_Y\le\theta+1\}.
$$
Then as $F\in\Fc_\theta$, it follows from the independence of the Brownian motion increments that
\begin{equation}\label{eq:one-step-estimate-preparation}
\P\big[F\cap\{\tau_Y\le\theta+1\}\big]
\ge
\P\Big[F\cap\Big\{\inf_{0\le s\le 1}(B_{\theta+s}-B_\theta)\le-\bar K\Big\}\Big]
=p_K\P[F],~~p_K:=\P\Big[\inf_{0\le s\le 1}B_s\le-\bar K\Big].
\end{equation}
Now define $A_K:=\{\tau_Y=\infty\}\cap\big\{\liminf_{t\to\infty}Y_t\le\frac{K}2\big\}$, and consider the stopping times
$$
\theta_1:=\inf\{t\ge0:\ Y_t\le K\},
~\theta_{m+1}:=\inf\{t\ge 1+\theta_m:\ Y_t\le K\},~m\ge 1.
$$
On the event set $A_K$, the process never hits $0$, and it visits $(0,K]$ infinitely many times. Then \(
A_K\subseteq\{\theta_m<\tau_Y\}$ and therefore $\P[A_K]\le\P[\theta_m<\tau_Y]$, for all $m\ge1$. Since $\theta_{m+1}\ge\theta_m+1$, we have
$
\{\theta_{m+1}<\tau_Y\}
\subset
\{\theta_m<\tau_Y\}\cap\{\tau_Y>\theta_m+1\}$, and then
\begin{align*}
\P[\theta_{m+1}<\tau_Y]
\le
\P\big[
\{\theta_m<\tau_Y\}\cap\{\tau_Y>\theta_m+1\}
\big]
&=
\P[\theta_m<\tau_Y]
-
\P\big[
\{\theta_m<\tau_Y\}\cap\{\tau_Y\le\theta_m+1\}
\big]  
\\
&\le
(1-p_K)\P[\theta_m<\tau_Y],
\end{align*}
where we applied the estimate
\eqref{eq:one-step-estimate-preparation} with $\theta=\theta_m$, together with the observation that $Y_{\theta_m}\le K$ on $\{\theta_m<\tau_Y\}$. Since $A_K\subseteq\{\sigma_m<\tau_Y\}$ for every $m\ge1$, we obtain
$
\P[A_K]\le(1-p_K)^{m-1}\longrightarrow0$ as $m\to\infty$,  which yields $\P[A_K]=0$. As $
\{\tau_Y=\infty\}
\cap
\left\{
\liminf_{t\to\infty}Y_t<\infty
\right\}
\subset
\bigcup_{K=1}^\infty A_K$, this yields
\[
\P\left[
\{\tau_Y=\infty\}
\cap
\left\{
\liminf_{t\to\infty}Y_t<\infty
\right\}
\right]=0
~\mbox{and then}~
\lim_{t\to\infty}Y_t=\infty,~\text{a.s. on}~\{\tau_Y=\infty\},
\]
which is the first part in \eqref{eq:fixed-R-preparation}. The second part follows immediately as this implies that
\begin{eqnarray*}
\mathds 1_{\{\tau_Y>t,Y_t\ge R\}}
&=&
\mathds 1_{\{\tau_Y=\infty\}\cap\{\tau_Y>t,Y_t\ge R\}}
+\mathds 1_{\{\tau_Y<\infty\}\cap\{\tau_Y>t,Y_t\ge R\}}
\\
&=&
\mathds 1_{\{\tau_Y=\infty\}\cap\{Y_t\ge R\}}
+\mathds 1_{\{\tau_Y<\infty\}\cap\{\tau_Y>t,Y_t\ge R\}}
\longrightarrow
\mathds 1_{\{\tau_Y=\infty\}},~~\mbox{a.s. for all}~R>0.
\end{eqnarray*}

\noindent  {\rm (ii)} We next prove \eqref{eq:remote-coordinate-estimate-general}. Let $F_y:=
\big\{\tau_Y>\sqrt{y}\}\cap\{Y_{\sqrt{y}}\ge R\big\}$, and observe that
\[
\P\big[(F_y)^c\big]
\le
\P\big[D_y\big] + \P\big[E_y\big],
~\text{where}~
D_y:=
\big\{\tau_Y\le\sqrt{y}\big\}
~\mbox{and}~ 
E_y:=
\big\{Y_{\sqrt{y}}<R\big\}.
\]
By Brownian scaling,
$$
\P[D_y]\le \P\left[
\inf_{0\le t\le \sqrt{y}}B_t\le -y+|\alpha|\sqrt{y}
\right]=
\P\left[
\inf_{0\le t\le1}B_t
\le
\frac{-y+|\alpha|\sqrt{y}}{y^{1/4}}
\right] 
\longrightarrow 0,
$$
and similarly,
\[
\P[E_y]
\le
\P\big[B_{\sqrt{y}}<R-y-\alpha \sqrt{y}\big]
\le
\P\Big[B_1<\frac{R-y+|\alpha|\sqrt{y}}{y^{1/4}}
\Big]
\longrightarrow 0,
\qquad\mbox{as }y\to\infty.
\]
Hence $\P[F_y]\longrightarrow 1$ as $y\to\infty$ as claimed in \eqref{eq:remote-coordinate-estimate-general}.
\ep
\end{proof}

\bibliographystyle{plain}
\bibliography{unique}

\end{document}